\pdfoutput=1
\RequirePackage{ifpdf}
\ifpdf % We are running pdfTeX in pdf mode
\documentclass[pdftex]{sigma}
\else
\documentclass{sigma}
\fi

\newcommand{\C}{\mathbb{C}}
\newcommand{\R}{\mathbb{R}}

\newcommand{\N}{\mathbb{N}}

\DeclareMathOperator{\tr}{tr}

\numberwithin{equation}{section}

\newtheorem{thm}{Theorem}[section]
\newtheorem{prop}[thm]{Proposition}
\newtheorem{lem}[thm]{Lemma}
\newtheorem{cor}[thm]{Corollary}
{ \theoremstyle{definition}
\newtheorem{defn}[thm]{Definition}
\newtheorem{ex}[thm]{Example}
 \newtheorem{rem}[thm]{Remark}}

\begin{document}

\allowdisplaybreaks

\newcommand{\arXivNumber}{1801.06013}

\renewcommand{\thefootnote}{}

\renewcommand{\PaperNumber}{109}

\FirstPageHeading

\ShortArticleName{Inverse of Infinite Hankel Moment Matrices}

\ArticleName{Inverse of Infinite Hankel Moment Matrices\footnote{This paper is a~contribution to the Special Issue on Orthogonal Polynomials, Special Functions and Applications (OPSFA14). The full collection is available at \href{https://www.emis.de/journals/SIGMA/OPSFA2017.html}{https://www.emis.de/journals/SIGMA/OPSFA2017.html}}}

\Author{Christian BERG~$^\dag$ and Ryszard SZWARC~$^\ddag$}

\AuthorNameForHeading{C.~Berg and R.~Szwarc}

\Address{$^\dag$~Department of Mathematical Sciences, University of Copenhagen,\\
\hphantom{$^\dag$}~Universitetsparken 5, DK-2100 Copenhagen, Denmark}
\EmailD{\href{mailto:berg@math.ku.dk}{berg@math.ku.dk}}
\URLaddressD{\url{http://www.math.ku.dk/~berg/}}

\Address{$^\ddag$~Institute of Mathematics, University of Wroclaw,\\
\hphantom{$^\ddag$}~pl.~Grunwaldzki 2/4, 50-384 Wroc{\l}aw, Poland}
\EmailD{\href{mailto:szwarc2@gmail.com}{szwarc2@gmail.com}}
\URLaddressD{\url{http://www.math.uni.wroc.pl/~szwarc/}}

\ArticleDates{Received January 19, 2018, in final form October 02, 2018; Published online October 06, 2018}

\Abstract{Let $(s_n)_{n\ge 0}$ denote an indeterminate Hamburger moment sequence and let $\mathcal H=\{s_{m+n}\}$ be the corresponding positive definite Hankel matrix. We consider the question if there exists an infinite symmetric matrix $\mathcal A=\{a_{j,k}\}$, which is an inverse of $\mathcal H$ in the sense that the matrix product $\mathcal A\mathcal H$ is defined by absolutely convergent series and $\mathcal A\mathcal H$ equals the identity matrix $\mathcal I$, a property called (aci). A candidate for $\mathcal A$ is the coefficient matrix of the reproducing kernel of the moment problem, considered as an entire function of two complex variables. We say that the moment problem has property (aci), if (aci) holds for this matrix $\mathcal A$. We show that this is true for many classical indeterminate moment problems but not for the symmetrized version of a cubic birth-and-death process studied by Valent and co-authors. We consider mainly symmetric indeterminate moment problems and give a number of sufficient conditions for (aci) to hold in terms of the recurrence coefficients for the orthonormal polynomials. A sufficient condition is a rapid increase of the recurrence coefficients in the sense that the quotient between consecutive terms is uniformly bounded by a constant strictly smaller than one.
We also give a simple example, where (aci) holds, but an inverse matrix of $\mathcal H$ is highly non-unique.}

\Classification{42C05; 44A60; 47B36; 33D45; 60J80}

\renewcommand{\thefootnote}{\arabic{footnote}}
\setcounter{footnote}{0}

\section{Introduction and summary of results}
Consider a Hamburger moment sequence $(s_n)$ given as
\begin{gather*}
s_n=\int_{-\infty}^\infty x^n{\rm d}\mu(x), \qquad n\ge 0,
\end{gather*}
 where $\mu$ is a probability measure with infinite support and moments of any order. The corresponding orthonormal polynomials $P_n$, obtained from the sequence of monomials $x^n$ by the Gram--Schmidt procedure, satisfy $P_0=1$ and
\begin{gather}\label{eq:rec}
xP_n(x)=b_nP_{n+1}(x)+a_nP_n(x)+b_{n-1}P_{n-1}(x),\qquad n\ge 0,
\end{gather}
where $b_n>0, $ $a_n\in\mathbb{R}$ with the convention $b_{-1}=1$, $P_{-1}=0$. For $n\ge 0$ we have $\deg P_n=n$ and the leading coefficient of $P_n$ is equal to
\begin{gather}\label{eq:lead}
{1\over b_0b_1\cdots b_{n-1}}.
\end{gather}

We are interested in the so-called indeterminate case, when the moments $s_n$ do not determine the measure $\mu$ uniquely. Such measures must have unbounded support. In particular the numbers $s_{2n}$ must grow faster than any exponential.

The infinite Hankel matrix $\mathcal H=\{h_{m,n}\}$
 \begin{gather*}
h_{m,n}=s_{m+n}=\big\langle x^m,x^n\big\rangle_{L^2(\mu)},\qquad m,n\ge 0
\end{gather*}
is positive definite in the sense that
\begin{gather}\label{eq:Hank}
(\mathcal Hv,v)=\sum_{m,n=0}^\infty s_{m+n}v_n\overline{v}_m>0,\qquad v\in \mathcal{F}_c(\mathbb{N}_0), \qquad v\neq 0,
\end{gather}
where $\mathcal{F}_c(\mathbb{N}_0)$ denotes the set of complex sequences $v=(v_n)_{n\ge 0}$, with only finitely many non-zero entries.

Berg, Chen and Ismail \cite{B:C:I} showed that the quadratic form associated with the matrix $\mathcal H$ is bounded below, i.e., for a suitable constant $C>0$
\begin{gather*}
(\mathcal{H}v,v)\ge C\|v\|^2,\qquad v\in \mathcal{F}_c(\mathbb{N}_0),
\end{gather*}
if and only if the moment problem associated with $(s_n)$ is indeterminate.

Hamburger obtained a related characterization of indeterminacy in terms of the quadratic forms associated with the matrices $\{s_{j+k}\}$ and $\{s_{j+k+2}\}$, see \cite[p.~70]{S:T}. For a comparison between Hamburger's result and the result of Berg--Chen--Ismail, see \cite{B:C:I}.

Yafaev \cite{Y} contains a characterization of closability of the Hankel form \eqref{eq:Hank} leading to the determinate case.

The smallest eigenvalue $\lambda_N$ of the truncated matrix $\mathcal{H}_N=\{s_{m+n}\}_{0\le m,n\le N}$ has been studied in a number of papers starting with Szeg\H{o} and Widom--Wilf, see~\cite{B:S1} and references therein.

The inverse of $\mathcal H_N$ can simply be described as the matrix
\begin{gather*}
\mathcal A^{(N)}=\{a_{j,k}(N)\}_{0\le j,k\le N},
\end{gather*}
where
\begin{gather}\label{eq:N-repker}
K_N(z,w)=\sum_{n=0}^NP_n(z)P_n(w)=\sum_{j,k=0}^N a_{j,k}(N)z^jw^k.
\end{gather}
A proof and historical details can be found in \cite[Theorem~1.1]{B:S1}.

In the indeterminate case it is possible to let $N\to \infty$ in \eqref{eq:N-repker} and obtain the infinite matrix $\mathcal A=\{a_{j,k}\}$ with $a_{j,k}=\lim\limits_{N\to\infty}a_{j,k}(N)$ as coefficient matrix of the entire function of two complex variables
\begin{gather*}
K(z,w)=\sum_{n=0}^\infty P_n(z)P_n(w)=\sum_{j,k=0}^\infty a_{j,k}z^jw^k,
\end{gather*}
called the reproducing kernel of the indeterminate moment problem.

The coefficients of $\mathcal H$ can grow fast, but on the other hand the coefficients of $\mathcal A$ decay fast, so it is natural to examine if it is possible to let $N$ tend to infinity in the equations
\begin{gather*}
\sum_{k=0}^N a_{i,k}(N)s_{k+j}=\delta_{i,j},\qquad 0\le i,j\le N,
\end{gather*}
and to obtain the infinite matrix equation
\begin{gather}\label{eq:AH}
\mathcal A\mathcal H=\mathcal H\mathcal A=\mathcal I,
\end{gather}
i.e.,
\begin{gather}\label{eq:AH1}
\sum_{k=0}^\infty a_{i,k}s_{k+j} =\delta_{i,j},\qquad i,j \ge 0,
\end{gather}
meaning that all the above series are convergent with the given sum.
Notice that since both matrices $\mathcal H$ and $\mathcal A$ are symmetric, it is enough to prove $\mathcal A\mathcal H=\mathcal I$ in~\eqref{eq:AH}.

Since absolute convergence is often more easy to establish than convergence, we investigate the question if the series in~\eqref{eq:AH1} are absolutely convergent for all~$i$,~$j$.

For the Stieltjes--Wigert moment sequence (sometimes called log-normal moments)
\begin{gather*}
s_n=q^{-\tfrac12 (n+1)^2},\qquad 0<q<1,
\end{gather*}
 it was shown that \eqref{eq:AH} holds and all the series in \eqref{eq:AH1} are absolutely convergent, see \cite[Section~5]{B:S1}. After this was done the same results were established but not published for several other known indeterminate moment problems, and we began to search for general theorems about~\eqref{eq:AH}. In this paper we present the general results obtained so far. At present we have no applications of equation~\eqref{eq:AH}.

If $\mathcal H$ is replaced by an infinite positive definite Toeplitz matrix $M$, the question of existence of a classical inverse of $M$ in the sense of \eqref{eq:AH} has been studied in~\cite{E:G:T}.

The sequence $c_k:=\sqrt{a_{k,k}}$ plays a crucial role in our investigations. It decays quickly to zero in the sense that $\lim\limits_{k\to\infty} k\root{k}\of{c_{k}}=0$, see~\eqref{eq:minexpt}.

Concerning the matrix $\mathcal A$, it was shown in \cite[Section~4]{B:S1} that it is of trace class, and that
\begin{gather*}
\tr\mathcal A=\frac{1}{2\pi}\int_0^{2\pi} \sum_{k=0}^\infty \big|P_k\big({\rm e}^{{\rm i}t}\big)\big|^2{\rm d}t.
\end{gather*}

In the indeterminate case the infinite Hankel matrix $\mathcal H$ does not correspond to an operator on $\ell^2$ defined on the subspace spanned by the standard orthonormal basis $(\delta_n)_{n\ge 0}$ in $\ell^2$, where $\delta_n=(\delta_{n,m})_{m\ge 0}$.

In fact, by a theorem of Carleman we necessarily have $\sum\limits_{n=0}^\infty s_{2n}^{-1/(2n)}<\infty$, hence $s_{2n}\ge 1$ for $n$ sufficiently large, and therefore
\begin{gather*}
\sum_{m=0}^\infty s_{n+m}^2=\infty\qquad \mbox{for all} \ n.
\end{gather*}

In connection with the study of \eqref{eq:AH} we introduce the following:

\begin{defn}\label{thm:defAH} An indeterminate moment problem has property
\begin{enumerate}\itemsep=0pt
\item[{\rm (c)}] if all the series in \eqref{eq:AH1} are convergent,
\item[{\rm (ac)}] if all the series in \eqref{eq:AH1} are absolutely convergent,
\item[{\rm (aci)}] if it has property {\rm (ac)} and \eqref{eq:AH} holds,
\item[{\rm (cs)}] if $\sum\limits_{n=0}^\infty c_{2n}s_{2n}<\infty$,
\item[{\rm (cs*)}] if $\sum\limits_{l=0}^\infty c_l|s_{l+m}|<\infty$ for all $m\ge 0$.
\end{enumerate}
\end{defn}

\begin{thm}\label{thm:start} For an indeterminate moment problem the following holds:
\begin{enumerate}\itemsep=0pt
\item[{\rm (i)}] {\rm (cs*)} $\implies$ {\rm (aci)} $\implies$ {\rm (ac)} $\implies$ {\rm (c)},
\item[{\rm (ii)}] {\rm (cs*)} $\implies$ {\rm (cs)}.
\end{enumerate}
\end{thm}

Concerning (i), the claim (cs*) $\implies$ (aci) is proved in Lemma~\ref{thm:suffAH}, the rest is obvious. To see (ii) we put $m=0$ in (cs*) and get
\begin{gather*}
\sum_{n=0}^\infty c_{2n}s_{2n}\le\sum_{l=0}^\infty c_l|s_l|<\infty.
\end{gather*}

We shall mainly focus on the case of symmetric moment problems, where the central recurrence coefficients $a_n=0$ for all~$n$, and the moment problem is characterized by the sequence~$(b_n)$.

\begin{rem}\label{thm:notreverse} We are not able to decide if some of these implications can be reversed. Note however, that for Stieltjes problems and for symmetric moment problems the properties (ac) and (aci) are equivalent, cf.\ Propositions~\ref{thm:stieltjesAH} and~\ref{thm:symindetac}.
\end{rem}

We next define some classes of sequences $(b_n)$ which will be considered in connection with the properties of Definition~\ref{thm:defAH}.

\begin{defn}\label{thm:deflogcvcc} A sequence $(b_n)$ of positive numbers is called eventually log-convex if there exists $n_0\in\mathbb N$ such that
\begin{gather*}
b_n^2\le b_{n-1}b_{n+1},\qquad n\ge n_0,
\end{gather*}
and it is called eventually log-concave if the reverse inequality holds. For $n_0=1$ we drop ``eventually''.
\end{defn}

The following lemma is given in \cite[Lemma~4.1]{B:S2}.

\begin{lem}\label{thm:lemBS2} Let $(b_n)$ be an unbounded positive sequence which is either eventually log-convex or eventually log-concave. Then $(b_n)$ is eventually strictly increasing to infinity.
\end{lem}

The following concept is crucial for several results.

\begin{defn}\label{thm:defqinc} A sequence $(b_n)$ of positive numbers is called eventually $q$-increasing, if there exist a number $0<q<1$ and $n_0\in\N$ such that
\begin{gather}\label{eq:qinc}
\frac{b_{n-1}}{b_n}\le q,\qquad n\ge n_0.
\end{gather}
For $n_0=1$ we drop ``eventually''.
\end{defn}

Note that if $(b_n)$ is unbounded and eventually log-convex, then it is eventually $q$-increasing with $q=b_{n_0-1}/b_{n_0}<1$, where $n_0$ is so large that $(b_n)$ is log-convex and strictly increasing for $n \ge n_0$.

We recall that an indeterminate moment problem has an order $\rho$ satisfying $0 \le\rho\le 1$. This is the common order of the four functions of the Nevanlinna matrix, see~\cite{B:P}.

\begin{prop}\label{thm:qinc-indet}\quad
\begin{enumerate}\itemsep=0pt
\item[$(i)$] The symmetric moment problem for an eventually $q$-increasing sequence $(b_n)$ is indeterminate and of order~$0$.

\item[$(ii)$] The symmetric moment problem for a positive sequence $(b_n)$ satisfying
\begin{gather}\label{eq:betan}
\frac{b_{n-1}}{b_n}\le {\rm e}^{-\beta/n},\qquad n\ge n_0, \qquad \beta>1,
\end{gather}
is indeterminate of order $\le 1/\beta$.
\end{enumerate}
\end{prop}

\begin{proof} Let $(Q_n)$ denote the sequence of polynomials of the second kind. By \cite[Theorem~1.2]{B:S2}, it suffices to prove that $\big(P_n^2(0)\big), \big(Q_n^2(0)\big)\in\ell^\alpha$ for all $0<\alpha\le 1$ in the first case and for all $\alpha>1/\beta$ in the second case.

 By the symmetry assumption we know that $P_{2n+1}(0)=Q_{2n}(0)=0$ and by \cite[Remark 4.5]{B:S2} we have
\begin{gather}\label{eq:P2n(0)}
P_{2n}(0)=(-1)^n\frac{b_0b_2\cdots b_{2n-2}}{b_1b_3\cdots b_{2n-1}},\qquad Q_{2n+1}(0)=(-1)^n\frac{b_1b_3\cdots b_{2n-1}}{b_0b_2\cdots b_{2n}}.
\end{gather}
For $n>n_0$ such that \eqref{eq:qinc} holds, we have
\begin{gather*}
P_{2n}^2(0)\le \left(\frac{b_0b_2\cdots b_{2n_0-2}}{b_1b_3\cdots b_{2n_0-1}}\right)^2q^{2(n-n_0)},\qquad
Q_{2n+1}^2(0)\le \left(\frac{b_1b_3\cdots b_{2n_0-1}}{b_0b_2\cdots b_{2n_0}}\right)^2 q^{2(n-n_0)},
\end{gather*}
and the first assertion follows.

For $n>n_0$ such that \eqref{eq:betan} holds, we have
\begin{gather*}
P_{2n}^2(0)\le \left(\frac{b_0b_2\cdots b_{2n_0-2}}{b_1b_3\cdots b_{2n_0-1}}\right)^2\exp\left(-2\beta\sum_{j=n_0}^{n-1}\frac{1}{2j+1}\right),
\end{gather*}
and a similar inequality for $Q_{2n+1}^2(0)$. Using the inequality
\begin{gather*}
\sum_{j=n_0}^{n-1}\frac{1}{2j+1}\ge \int_{n_0}^n\frac{{\rm d}x}{2x+1}=\frac12\log\frac{2n+1}{2n_0+1},
\end{gather*}
we get
\begin{gather*}
P_{2n}^2(0)\le \frac{C}{(2n+1)^\beta}
\end{gather*}
for a suitable constant $C$, and the second assertion follows.
\end{proof}

The following examples have $q$-increasing coefficients $b_n$.

\begin{ex}\label{thm:ops1} The discrete $q$-Hermite polynomials of type II are given by
\begin{gather*}
x\tilde{h}_n(x;q)=\tilde{h}_{n+1}(x;q)+q^{-2n+1}\big(1-q^n\big)\tilde{h}_{n-1}(x;q),
\end{gather*}
cf.~\cite{K:S}, where $0<q<1$. This is the monic version and the corresponding coefficients for the orthonormal version are
\begin{gather*}
b_n=q^{-n-1/2}\sqrt{1-q^{n+1}}.
\end{gather*}
It follows that
\begin{gather*}
\frac{b_{n-1}}{b_n}=q\sqrt{\frac{1-q^n}{1-q^{n+1}}}<q,
\end{gather*}
so $(b_n)$ is $q$-increasing and log-concave.

It is known that $s_{2n}=\big(q;q^2\big)_nq^{-n^2}$, where the shifted factorials are defined in \eqref{eq:qshift}.

In \cite{C:I} Christiansen and Ismail studied symmetric Al-Salam--Chihara polynomials of type II given by the recursion
\begin{gather*}
2xp^nQ_n(x;\beta)=\big(1-p^{n+1}\big)Q_{n+1}(x;\beta)+\big(\beta+p^{n-1}\big)Q_{n-1}(x;\beta),
\end{gather*}
where $0<p<1$, $\beta\ge 0$. The recurrence coefficients for the orthonormal version are
\begin{gather*}
b_n=\tfrac{1}{2}p^{-n-1/2}\big[\big(1-p^{n+1}\big)\big(\beta+p^n\big)\big]^{1/2}.
\end{gather*}
It follows that
\begin{gather*}
\frac{b_{n-1}}{b_n}=p\left[\frac{1-p^n}{1-p^{n+1}}\frac{\beta+p^{n-1}}{\beta+p^n}\right]^{1/2}\le p\left[\frac{\beta+1}{\beta+p}\right]^{1/2}\le\sqrt{p},
\end{gather*}
so $(b_n)$ is $\sqrt{p}$-increasing.

One can prove that $(b_n)$ is eventually log-convex when $0<\beta<1/p$ and eventually log-concave when $\beta\ge 1/p$ or $\beta=0$. The case $\beta=0$ yields the $p^{-1}$-Hermite polynomials studied by Ismail and Masson in \cite{I:M}.
\end{ex}

We shall now formulate our main results for symmetric moment problems given in terms of $a_n=0$ and $b_n>0$ from \eqref{eq:rec}. Our results are based on a study of two quantities $U_n$, $V_n$ defined in Section~\ref{section2}, see~\eqref{eq:Un} and~\eqref{eq:Vn}. As far as we know these quantities have not been investigated previously. They are treated in Sections~\ref{section3} and~\ref{section4}, and to avoid duplication we shall not mention these technical results here.

We need the following decreasing functions
\begin{gather}\label{eq:ua}
u(\alpha):=\max_{2^{-1/\alpha}\le x\le 1} (x+1)\int_x^1\log\left(\frac{y^\alpha}{1-y^\alpha}\right)\frac{{\rm d}y}{(1+y)^2},\qquad \alpha>0, \\
\label{eq:va}
v(\alpha):=\int_1^\infty \log\big(\big[1-x^{-\alpha}\big]^{-1}\big){\rm d}x,\qquad \alpha>1, \\
\label{eq:ka}
k(\alpha):=4^{-\alpha}\exp(2u(\alpha)+v(\alpha)),\qquad \alpha>1,
\end{gather}
which are motivated by Theorems~\ref{thm:thm2},~\ref{thm:genlogcc}, and~\ref{thm:evtlogconcave} respectively. The limiting values of these functions at the end-points of their interval of definition are also given there. Further properties of the functions $u$, $v$ are given in Appendix~\ref{appendixA}.

\begin{thm}\label{thm:Main3} Consider a symmetric moment problem given in terms of an eventually $q$-increasing sequence $(b_n)$, cf.\ Definition~{\rm \ref{thm:defqinc}}. We then have
\begin{enumerate}\itemsep=0pt
\item[$(i)$] $(U_n)$, $(V_n)$, $\big(c_n\sqrt{s_{2n}}\big)$ are bounded sequences.
\item[$(ii)$] Property {\rm (cs*)} holds.
\end{enumerate}
\end{thm}

This result follows from Theorems~\ref{thm:evtlogconvex} and~\ref{thm:Vevtlogconv}, Corollary~\ref{thm:VnsqrtUn} and Theorem~\ref{thm:AH-logcv2}.

We next relax the condition of $(b_n)$ being eventually $q$-increasing.

\begin{thm}\label{thm:Main4} Consider a symmetric moment problem given in terms of $(b_n)$ satisfying
\begin{gather*}
\frac{b_{n-1}}{b_n}\le {\rm e}^{-f(n)},\qquad n\ge n_0,
\end{gather*}
where $f(n)>0$ for $n\ge n_0$ and $\alpha:=\liminf\limits_{n\to\infty} nf(n)\in [0,\infty]$.

We then have
\begin{enumerate}\itemsep=0pt
\item[$(i)$] $\limsup\limits_{n\to\infty}\root{n}\of{U_n}\le {\rm e}^{2u(\alpha)}$, $0\le \alpha\le\infty$,
\item[$(ii)$] $\limsup\limits_{n\to\infty}\root{n}\of{V_n}\le {\rm e}^{v(\alpha)/2}$, $1\le \alpha\le\infty$,
\item[$(iii)$] Property {\rm (cs*)} holds if $k(\alpha)<1$, which in turn holds for $\alpha\ge 1.68746$, cf.\ Remark~{\rm \ref{thm:rem4}}.
\end{enumerate}
\end{thm}

This result follows from Theorems~\ref{thm:thm2a},~\ref{thm:thm2b}, and~\ref{thm:evtlogconcave}.

The last theorem applied to the symmetric indeterminate moment problem with $b_n=(n+1)^c$, $c>1$, yields that~(cs*) holds for $c\ge 1.68746$. This is not optimal, because in Section~\ref{section7} we prove by completely different methods that~(cs*) actually holds for $c>3/2$ and that property (cs) (and hence (cs*)) does not hold for $1<c<3/2$. The behaviour for the border case $c=3/2$ is not known.

A main difficulty is that almost nothing is known about the orthonormal polynomials and related quantities in case $b_n=(n+1)^c$, see the preamble to Section~\ref{section7}. We have relied on the symmetrized version of a birth-and-death process with cubic rates studied by Valent in~\cite{Va98}. In Section~\ref{section6} we find an estimate of the moments and explicit expressions for the entries of the matrix~$\mathcal A$ in that case.

A symmetric moment problem given by an eventually log-convex or eventually log-concave sequence $(b_n)$ is indeterminate if and only if the Carleman condition $\sum 1/b_n<\infty$ holds. In the indeterminate case the order $\rho$ of the moment problem is equal to the exponent of convergence~$\mathcal E(b_n)$ of~$(b_n)$ defined as
\begin{gather*}
\mathcal E(b_n)=\inf\left\{\alpha>0 \colon \sum_{n=0}^\infty \frac{1}{b_n^\alpha}<\infty\right\}.
\end{gather*}
See Theorems 1.4 and~4.11 in~\cite{B:S2}. The log-concave case builds on technique introduced by Berezanski\u{\i}, see~\cite[p.~26]{Ak}.

In Appendix~\ref{appendixA} we take up certain questions related to the above, but which are not needed in the main results. We prove that $\liminf_n U_n^{1/n}\ge 4$ for all bounded sequences~$(b_n)$. In Remark~\ref{thm:nonuniq2} we notice that although~$\mathcal A$ can be an inverse of~$\mathcal H$, it is not uniquely determined by being so. We also give some properties of the functions~$u$,~$v$.

\section{Preliminaries}\label{section2}
In this section we shall summarize some notation and results from our papers \cite{B:S1,B:S2,B:S3}, which will be needed in the following. We use the following notation for the orthonormal polynomials
\begin{gather}%\label{eq:p}
P_n(x)=b_{n,n}x^n+b_{n-1,n}x^{n-1}+\cdots + b_{1,n}x+b_{0,n},\nonumber\\
x^n=c_{n,n}P_n(x)+c_{n-1,n}P_{n-1}(x)+\cdots + c_{1,n}P_1(x)+c_{0,n}P_0(x).\label{eq:x}
\end{gather}
By (\ref{eq:lead}) we get
\begin{gather}\label{eq:b_n} b_{n,n}={1\over b_0b_1\cdots b_{n-1}},\qquad c_{n,n}=b_0b_1\cdots b_{n-1}.
\end{gather}
The matrices $\mathcal B=\{b_{i,j}\}$ and $\mathcal C=\{c_{i,j}\}$ with the assumption
\begin{gather*}
b_{i,j}=c_{i,j}=0\qquad \text{for} \quad i>j
\end{gather*}
 are upper-triangular. Since $\mathcal B$ and $\mathcal C$ are transition matrices between two sequences of linearly independent systems of functions, we have
\begin{gather*}
\mathcal B\mathcal C=\mathcal C\mathcal B=\mathcal I.
\end{gather*}
Observe that (for say $m\le n$)
\begin{gather*}s_{m+n} = h_{m,n}=\big\langle x^m,x^n\big\rangle_{L^2(\mu)} = c_{0,m}c_{0,n}+c_{1,m}c_{1,n}+\cdots +c_{m,m}c_{m,n} =\big(\mathcal C^t \mathcal C\big)_{m,n},
\end{gather*}
so $\mathcal{H}=\mathcal C^t \mathcal C$.

Hence we have
\begin{gather}\label{eq:I=B(B^tH)}
\mathcal I=\mathcal B\mathcal C=\mathcal B\big(\mathcal B^t\mathcal C^t\mathcal C\big)=\mathcal B\big(\mathcal B^t\mathcal{H}\big).
\end{gather}
The multiplication is well defined as $\mathcal B$, $\mathcal C$ are upper-triangular matrices and $\mathcal B^t$, $\mathcal C^t$ are lower-triangular, so we are always dealing with finite sums in the matrix products.

Let us now assume that the moment problem is indeterminate. We recall from \cite[Proposition~4.2]{B:S1} that $\mathcal B$ is a Hilbert--Schmidt matrix and that \begin{gather}\label{eq:A=BBt}
\mathcal{A}=\mathcal B\mathcal B^t,
\end{gather}
 This means
\begin{gather}\label{eq:akl}
a_{k,l}=\sum^\infty_{n=\max(k,l)}b_{k,n}b_{l,n},
\end{gather}
hence
\begin{gather*}
a_{k,l}^2\le \left (\sum_{n=k}^\infty b_{k,n}^2\right ) \left (\sum_{n=l}^\infty b_{l,n}^2\right ).\end{gather*}
Let us recall some of the arguments leading to these results. Because of indeterminacy it
is known that the series
\begin{gather*}
\sum_{n=0}^\infty |P_n(z)|^2
\end{gather*}
 is convergent uniformly for $z$ in bounded subsets of $\mathbb{C}$. Moreover, the sum has subexponential growth by a result of M. Riesz, cf. \cite[p. 56]{Ak}, i.e., it is dominated by a multiple of ${\rm e}^{\varepsilon |z|}$ for any $\varepsilon>0$.
Fixing $r>0$, we have
\begin{gather*}
P_n\big(r{\rm e}^{{\rm i}t}\big)=\sum_{k=0}^n b_{k,n}r^k{\rm e}^{{\rm i}tk}.
\end{gather*}
By Parseval's identity we obtain
\begin{gather*}
{1\over 2\pi}\int_0^{2\pi}\big|P_n\big(r{\rm e}^{{\rm i}t}\big)\big|^2{\rm d}t = \sum_{k=0}^n b_{k,n}^2r^{2k}.
\end{gather*}
 Next
\begin{gather*}
{1\over 2\pi}\int_0^{2\pi}\sum_{n=0}^\infty \big|P_n\big(r{\rm e}^{{\rm i}t}\big)\big|^2{\rm d}t=\sum_{n=0}^\infty\sum_{k=0}^n b_{k,n}^2r^{2k}=\sum_{k=0}^\infty r^{2k}\sum_{n=k}^\infty b_{k,n}^2.
\end{gather*}
Therefore, denoting
\begin{gather}\label{eq:c_k}
c_k=\left (\sum_{n=k}^\infty b_{k,n}^2\right )^{1/2}=\sqrt{a_{k,k}},
\end{gather}
 we have
\begin{gather*}\label{eq:r}
\sum_{k=0}^\infty c_k^2r^{2k}={1\over 2\pi}\int_0^{2\pi}\sum_{n=0}^\infty \big|P_n\big(r{\rm e}^{{\rm i}t}\big)\big|^2{\rm d}t.
\end{gather*}
(For $r=1$ this proves that $\mathcal B$ is Hilbert--Schmidt, and in particular the series in~\eqref{eq:akl} is absolutely convergent.)

\begin{lem}\label{thm:suffAH} Consider an indeterminate moment problem. If
\begin{gather}\label{eq:suffAH}
\sum_{n,l=0}^\infty |b_{k,n}b_{l,n}s_{l+m}|<\infty\qquad \mbox{for all}\quad k,m\ge 0,
\end{gather}
then property {\rm (aci)} from Definition~{\rm \ref{thm:defAH}} holds.

For \eqref{eq:suffAH} to hold it suffices that
\begin{gather}\label{eq:suffAH1}
\sum_{l=0}^\infty c_l |s_{l+m}|<\infty\qquad\mbox{for all}\quad m\ge 0.
\end{gather}
In particular {\rm (cs*)} implies {\rm (aci)}.
\end{lem}

\begin{proof} We first notice that \eqref{eq:suffAH} implies (ac) because of \eqref{eq:akl}. We next note that \eqref{eq:suffAH} implies
\begin{gather*}
\mathcal B\big(\mathcal B^t\mathcal H\big)=\big(\mathcal B\mathcal B^t\big)\mathcal H,
\end{gather*}
so by \eqref{eq:I=B(B^tH)} and \eqref{eq:A=BBt} we get
\begin{gather*}
\mathcal I=\mathcal B\big(\mathcal B^t\mathcal H\big)=\big(\mathcal B\mathcal B^t\big)\mathcal H=\mathcal A\mathcal H.
\end{gather*}
Suppose next that \eqref{eq:suffAH1} holds. By the Cauchy--Schwarz inequality
\begin{gather*}
\sum_{n,l=0}^\infty |b_{k,n}b_{l,n}s_{l+m}|\le \sum_{l=0}^\infty c_kc_l|s_{l+m}|<\infty,
\end{gather*}
where we have used \eqref{eq:c_k}.
\end{proof}

\begin{rem}\label{thm:c-sqrt} It is tempting to use $|s_{l+m}|\le \sqrt{s_{2l}}\sqrt{s_{2m}}$ to assert that the conditions of Lem\-ma~\ref{thm:suffAH} are satisfied if $\sum c_l\sqrt{s_{2l}}<\infty$. This is not possible because $c_l\sqrt{s_{2l}}\ge 1$ by Lemma~\ref{thm:Un}.
\end{rem}

We shall mainly be concerned about \eqref{eq:AH} in the symmetric case, characterized by $a_n\equiv 0$, but since there is a one-to-one correspondence between Stieltjes moment problems and symmetric moment problems, we shall first compare \eqref{eq:AH} in the two cases.

Suppose
\begin{gather*}
t_n=\int_0^\infty t^n{\rm d}\nu(t),\qquad n\ge 0
\end{gather*}
is a Stieltjes moment sequence of a probability measure $\nu$ on $[0,\infty[$ with infinite support. The corresponding orthonormal polynomials $(P_n)$ satisfy~\eqref{eq:rec}, where now $a_n>0$. Let $(R_n)$ denote the orthonormal polynomials with respect to the measure $t{\rm d}\nu(t)$. Then
\begin{gather*}
S_{2n}(x):=P_n\big(x^2\big),\qquad S_{2n+1}(x):=xR_n\big(x^2\big)
\end{gather*}
are the orthonormal polynomials with respect to the unique symmetric probability measure~$\mu$ on $\mathbb R$ such that for all non-negative Borel functions $f\colon [0,\infty[\to\mathbb R$
\begin{gather*}
\int_{-\infty}^\infty f\big(x^2\big){\rm d}\mu(x)=\int_0^\infty f(t){\rm d}\nu(t).
\end{gather*}
The moments $s_n$ of $\mu$ are $s_{2n}=t_n$, $s_{2n+1}=0$.

The three terms recurrence relation for $(S_n)$ has the form
\begin{gather}\label{eq:recsym(s)}
xS_n(x)={}^sb_ n S_{n+1}(x)+{}^sb_{n-1}S_{n-1}(x)
\end{gather}
with
\begin{gather*}
a_0=({}^sb_0)^2,\qquad b_n=({}^sb_{2n})({}^sb_{2n+1}),\qquad n\ge 0, \qquad a_n=({}^sb_{2n})^2+({}^s b_{2n-1})^2,\qquad n\ge 1.
\end{gather*}

It is known that $\mu$ is indeterminate if and only if $\nu$ is Stieltjes indeterminate, i.e., there are measures on $[0,\infty[$ different from $\nu$ with the same moments as~$\nu$, cf.~\cite[p.~333]{Ch}. In the affirmative case we denote the reproducing kernels
\begin{gather*}
K(z,w)=\sum_{n=0}^\infty P_n(z)P_n(w)=\sum_{j,k=0}^\infty a_{j,k}z^jw^k,\\
K'(z,w)=\sum_{n=0}^\infty R_n(z)R_n(w)=\sum_{j,k=0}^\infty a'_{j,k}z^jw^k.
\end{gather*}
(Notice that also $t{\rm d}\nu(t)$ is Stieltjes indeterminate.)

The reproducing kernel for the symmetric measure $\mu$ is given as
\begin{gather}
K_s(z,w) = \sum_{n=0}^\infty S_n(z)S_n(w)=\sum_{n=0}^\infty P_n\big(z^2\big)P_n\big(w^2\big)+zw\sum_{n=0}^\infty R_n\big(z^2\big)R_n\big(w^2\big)\nonumber \\
\hphantom{K_s(z,w)}{} = \sum_{j,k=0}^\infty a_{j,k}z^{2j}w^{2k}+\sum_{j,k=0}^\infty a'_{j,k}z^{2j+1}w^{2k+1} =\sum_{j,k=0}^\infty a^{(s)}_{j,k}z^{j}w^{k},\label{eq:repksym}
\end{gather}
where
\begin{gather*}
a_{j,k}^{(s)}=\begin{cases} a_{j/2,k/2}, & j, \ k \ \text{even},\\
0, & j+k \ \text{odd},\\
a'_{(j-1)/2,(k-1)/2}, & j, \ k \ \text{odd}.
\end{cases}
\end{gather*}
We then easily get the following about property (aci) from Definition~\ref{thm:defAH}:

\begin{prop}\label{thm:st-sym} The symmetric moment problem with moments $(s_n)$ has property {\rm (aci)} if and only if the two Stieltjes problems with moments
$(t_n)$ and $(t_{n+1})$ have property {\rm (aci)}.
\end{prop}

For a Stieltjes problem we notice that the coefficients $b_{k,n}$ of the orthonormal polynomials $P_n$ have the sign pattern
$b_{k,n}=(-1)^{n-k}|b_{k,n}|$ because the zeros of $P_n$ are all positive. This implies that
\begin{gather*}
a_{k,l}=\sum_{n=0}^\infty b_{k,n}b_{l,n}=(-1)^{k+l}\sum_{n=0}^\infty |b_{k,n}||b_{l,n}|,
\end{gather*}
hence
\begin{gather*}
\sum_{l=0}^\infty |a_{k,l}s_{l+m}|=\sum_{l,n=0}^\infty |b_{k,n}||b_{l,n}|s_{l+m}.
\end{gather*}

By Lemma~\ref{thm:suffAH} this shows

\begin{prop}\label{thm:stieltjesAH} For an indeterminate Stieltjes problem the properties {\rm (ac)} and {\rm (aci)} from Definition~{\rm \ref{thm:defAH}} are equivalent.
\end{prop}

Combining this with Proposition~\ref{thm:st-sym} we immediately get:

\begin{prop}\label{thm:symindetac} For a symmetric indeterminate moment problem the pro\-perties {\rm (ac)} and {\rm (aci)} from Definition~{\rm \ref{thm:defAH}} are equivalent.
\end{prop}

In order to verify convergence of the series in \eqref{eq:suffAH1}, we shall estimate the sequences $(c_n)$, defined in~\eqref{eq:c_k}, and $(s_{2n})$ in terms of the quantities
\begin{gather}\label{eq:Un}
U_n:=\frac{s_{2n}}{b_0^2b_1^2\cdots b_{n-1}^2},\qquad n\ge 1,\qquad U_0:=s_0=1
\end{gather}
and
\begin{gather}\label{eq:Vn}
V_n:=b_0 b_1\cdots b_{n-1}c_n,\qquad n\ge 1,\qquad V_0:=c_0>1.
\end{gather}
This will be done in the next two sections.

The quantity $U_n$ is defined for any moment problem, while $V_n$ only makes sense for indeterminate problems.

It is easy to see that both quantities are scale invariant in the sense that if $(a_n)$, $(b_n)$ from the recurrence relation \eqref{eq:rec} are replaced by a constant multiple $(\lambda a_n)$, $(\lambda b_n)$ for some $\lambda>0$, then $U_n$, $V_n$ remain unchanged.

In the following we shall often compare quantities like $s_n,U_n,c_n,\ldots$ depending on the recurrence coefficients $(a_n)$, $(b_n)$ with those corresponding to other recurrence coefficients, say $(\tilde{a}_n)$,~$\big(\tilde{b}_n\big)$. We then denote the corresponding quantities $\tilde{s}_n,\tilde{U}_n,\tilde{c}_n,\ldots$.

\begin{lem}\label{thm:Un}\quad
\begin{enumerate}\itemsep=0pt
\item[$(i)$] For any moment problem we have $U_n\ge 1$.
 \item[$(ii)$] For any indeterminate moment problem we have $V_n\ge 1$ and $V_n\sqrt{U_n}=c_n\sqrt{s_{2n}}\ge 1$.
\end{enumerate}
\end{lem}

\begin{proof}By \eqref{eq:x} we have
\begin{gather*}
s_{2n}=\langle x^n,x^n\rangle=c_{0,n}^2+c_{1,n}^2+\dots +c_{n,n}^2 \ge c_{n,n}^2= b_0^2b_1^2\cdots b_{n-1}^2,
\end{gather*}
hence by \eqref{eq:c_k} and \eqref{eq:b_n}
\begin{gather*}
c_n=\left (\sum_{j=n}^\infty b_{n,j}^2\right )^{1/2}\ge b_{n,n}={1\over b_0b_1\cdots b_{n-1}} \ge 1/\sqrt{s_{2n}}.\tag*{\qed}
\end{gather*}\renewcommand{\qed}{}
\end{proof}

In \cite[Section 3]{B:S2} we considered the power series
\begin{gather*}
\Phi(z)=\sum_{n=0}^\infty c_nz^n,
\end{gather*}
and proved that $\Phi$ is an entire function of minimal exponential type. In terms of $(c_n)$ this
can be expressed
\begin{gather}\label{eq:minexpt}
\lim_{k\to\infty} k c_k^{1/k} =0.
\end{gather}
It was also proved there that $\Phi$ has the same order $\rho$ and type $\tau$ as the moment problem, which by definition is the common order and type of the entire functions in the Nevanlinna matrix, see \cite{B:P,B:P:H}. From the general formula for type of an entire function in terms of its power series coefficients, see~\cite{Le}, we then get the following result.

\begin{thm}\label{thm:otcn} Consider an indeterminate moment problem of order $0<\rho<1$ and type $0\le \tau\le \infty$. Then
\begin{gather}\label{eq:otcn}
\limsup_{k\to\infty} k^{1/\rho} c_k^{1/k} = ({\rm e} \rho \tau )^{1/\rho}.
\end{gather}
\end{thm}

\section[Estimates of the even moments via $U_n$]{Estimates of the even moments via $\boldsymbol{U_n}$}\label{section3}
In all of this section we assume that the moment problem is symmetric or equivalently that $a_n\equiv 0$, so the recurrence relation~\eqref{eq:rec} takes the simplified form
\begin{gather}\label{eq:recsym}
xP_n=b_nP_{n+1}+b_{n-1}P_{n-1}.
\end{gather}

\begin{prop}\label{thm:increasofbn}Let $(b_n)$ and $\big(\tilde{b}_n\big)$ denote two sequences of positive numbers satisfying $b_n\le \tilde{b}_n$, $n\ge 0$. Then the corresponding even moments satisfy $s_{2n}\le \tilde{s}_{2n}$, $n\ge 0$.
\end{prop}

\begin{proof}Let $J$ denote the Jacobi matrix associated to the moment problem, i.e.,
\begin{gather}\label{eq:Jacobimat}
J=\begin{pmatrix} 0 & b_0 & 0 & 0 & \cdots\\
 b_0 & 0 & b_1 & 0 & \cdots\\
 0 & b_1 & 0 & b_2 & \cdots\\
 0 & 0 & b_2 & 0 & \cdots\\
 \vdots & \vdots & \vdots & \vdots & \ddots
\end{pmatrix}.
\end{gather}
It is well-known that $s_n=(J^n\delta_0,\delta_0)$, and the conclusion follows.
\end{proof}

In the symmetric case $P_k$ involves only monomials with the same parity as $k$. Thus we have
 \begin{gather}\label{eq:u0}
{x^n\over b_0b_1\cdots b_{n-1}}=\sum_{k=0}^{[n/2]}u_{n,k}P_{n-2k}=u_{n,0}P_n+u_{n,1}P_{n-2}+\cdots,
\end{gather}
 hence
 \begin{gather}\label{eq:u1}
U_n:= {s_{2n}\over b_0^2b_1^2\cdots b_{n-1}^2}=\sum_{k=0}^{[n/2]}u_{n,k}^2.
 \end{gather}
(Compared with \eqref{eq:x} we have $u_{n,k}=c_{n-2k,n}/(b_0\cdots b_{n-1})$, while $c_{n-2k-1,n}=0$.) By \eqref{eq:lead} we have $u_{n,0}=1$.

By the recurrence relation \eqref{eq:recsym} the coefficients $u_{n,k}$ satisfy
\begin{gather}\label{eq:u}
u_{n+1,k}={b_{n-2k}\over b_n} u_{n,k}+{b_{n+1-2k}\over b_n}u_{n,k-1},\qquad 1\le k\le [n/2].
\end{gather}
In fact, multiplying \eqref{eq:u0} by $x/b_n$ and using the recurrence relation \eqref{eq:recsym}, we get
\begin{gather*}
\frac{x^{n+1}}{b_0b_1\cdots b_n}=\sum_{k=0}^{[n/2]}\frac{u_{n,k}}{b_n}\left(b_{n-2k}P_{n-2k+1}+b_{n-2k-1}P_{n-2k-1}\right),
\end{gather*}
and \eqref{eq:u} follows. We may extend the definition by setting $u_{n,k}=0$ for $k>[n/2]$ or $k<0$ and $b_n=0$ for $n<0$. In this way it is not difficult to see that
 \begin{gather}\label{eq:u'}
u_{n+1,k}={b_{n-2k}\over b_n} u_{n,k}+{b_{n+1-2k}\over b_n}u_{n,k-1},\qquad n, k\ge 0,
\end{gather}
and it follows by induction that all coefficients $u_{n,k}\ge 0$.

By \eqref{eq:u1} and \eqref{eq:u} this immediately implies the following.

\begin{prop}\label{thm:comparison} Assume $(b_n)$ and $\big(\tilde{b}_n\big)$ are positive sequences satisfying
\begin{gather*}
{b_{n-1}\over b_{n}} \le {\tilde{b}_{n-1}\over \tilde{b}_{n}} ,\qquad n\ge 1,
\end{gather*}
 and let $s_{2n}$ and $\tilde{s}_{2n}$ denote the corresponding even moments. Then
\begin{gather*}
 U_n:={s_{2n}\over b_0^2b_1^2\cdots b_{n-1}^2} \le \tilde{U}_n:={\tilde{s}_{2n}\over \tilde{b}_0^2\tilde{b}_1^2\cdots \tilde{b}_{n-1}^2}.
\end{gather*}
\end{prop}

For $n_0\in\mathbb N$ the shifted recurrence sequence $b_n^{(n_0)}:=b_{n+n_0}$, $n\ge 0$ leads to moments $s_n^{(n_0)}$ and
\begin{gather*}
U_n^{(n_0)}:=\frac{s_{2n}^{(n_0)}}{\big(b_0^{(n_0)}\cdots b_{n-1}^{(n_0)}\big)^2}=\frac{s_{2n}^{(n_0)}}{b_{n_0}^2\cdots b_{n_0+n-1}^2}.
\end{gather*}

\begin{prop}\label{thm:shiftU} The expressions above satisfy
\begin{gather}\label{eq:shiftU1}
s_{2n+2}\ge b_0^2 s_{2n}^{(1)},\qquad U_{n+1}\ge U_n^{(1)}.
\end{gather}
If $b_n\le b_n^{(n_0)}$ for $n\ge 0$ and $\lim b_n^{1/n}=1$, then
\begin{gather}\label{eq:shiftU2}
\limsup_{n\to\infty} U_n^{1/n}=\limsup_{n\to\infty}\big(U_n^{(n_0)}\big)^{1/n}.
\end{gather}
The latter holds for $b_n=\lambda (n+1)^c$, $\lambda, c>0$.
\end{prop}

\begin{proof} Define
\begin{gather*}
J_{(1)}=\begin{pmatrix} 0 & 0 & 0 & 0 & \cdots\\
 0 & 0 & b_1 & 0 & \cdots\\
 0 & b_1 & 0 & b_2 & \cdots\\
 0 & 0 & b_2 & 0 & \cdots\\
 \vdots & \vdots & \vdots & \vdots & \ddots
\end{pmatrix}=
\begin{pmatrix} 0 & 0\\
0 & J^{(1)}
\end{pmatrix},
\end{gather*}
where $J^{(1)}$ is the Jacobi matrix \eqref{eq:Jacobimat} corresponding to $b_n^{(1)}$ and the matrix to the right is a block matrix. The upper left $0$ is a scalar and the upper right $0$ is an infinite row vector of zeros. In block matrix notation we then have
\begin{gather*}
\big(J_{(1)}\big)^n=\begin{pmatrix} 0 & 0\\
0 & \big(J^{(1)}\big)^n
\end{pmatrix},
\end{gather*}
hence
\begin{gather*}
s_{2n}^{(1)}=\big((J_{(1)})^{2n} \delta_1,\delta_1\big).
\end{gather*}
This gives
\begin{gather*}
s_{2n+2} = \big(J^{2n+2}\delta_0,\delta_0\big)=\big(J^{2n}J\delta_0,J\delta_0\big)=b_0^2\big(J^{2n}\delta_1,\delta_1\big)
\ge b_0^2 \big((J_{(1)})^{2n} \delta_1,\delta_1\big)=b_0^2s_{2n}^{(1)},
\end{gather*}
hence $U_{n+1}\ge U_n^{(1)}$.

If $b_n\le b_n^{(n_0)}$ we have $s_{2n}\le s_{2n}^{(n_0)}$ by Proposition~\ref{thm:increasofbn}, and therefore for $n>n_0$
\begin{gather*}
U_n\le \frac{b_{n_0}^2b_{n_0+1}^2\cdots b_{n_0+n-1}^2}{b_0^2 b_1^2\cdots b_{n-1}^2}U_n^{(n_0)}\le \frac{b_n^2b_{n+1}^2\cdots b_{n+n_0-1}^2}{b_0^2 b_1^2\cdots b_{n_0-1}^2}U_{n+n_0}.
\end{gather*}
If $\lim b_n^{1/n}=1$ we get \eqref{eq:shiftU2}.
\end{proof}

\begin{lem}\label{thm:b_i/b_{i+1}} Assume $(b_n)$ satisfies
\begin{gather*}
\sup_{i,j\ge 0}\frac{b_i}{b_{i+j}}=M<\infty.
\end{gather*}
Then
\begin{gather}\label{eq:Un+1}
U_{n+1}\le 4M^2U_n,\qquad s_{2n+2}\le 4M^2b_n^2s_{2n}.
\end{gather}
\end{lem}

\begin{proof} By the recurrence equation \eqref{eq:u} we get
\begin{gather*}
u_{n+1,k}\le M(u_{n,k}+u_{n,k-1}),
\end{gather*}
hence
\begin{gather*}
u_{n+1,k}^2\le 2M^2\big(u_{n,k}^2+u_{n,k-1}^2\big).
\end{gather*}
This gives
\begin{gather*}
U_{n+1}=\sum_{k=0}^{[(n+1)/2]} u_{n+1,k}^2\le 4M^2 U_n,
\end{gather*}
hence $s_{2n+2}\le 4M^2b_n^2s_{2n}$.
\end{proof}

In the following we shall use the notation from \cite{G:R} about $q$-shifted factorials. For $z,q\in\mathbb C$ we write
\begin{gather}\label{eq:qshift}
(z;q)_n=\prod_{j=0}^{n-1}\big(1-zq^{j}\big),\qquad n=1,2,\ldots,\qquad (z;q)_0=1.
\end{gather}

When $|q|<1$ we define the absolutely convergent infinite product
\begin{gather}\label{eq:q-infty}
(z;q)_{\infty}=\prod_{j=0}^\infty\big(1-zq^j\big).
\end{gather}

\begin{lem}\label{thm:logconcave} Assume $(b_n)$ is log-concave and strictly increasing. Define $q_n:=b_{n-1}/b_n$, $n\ge 1$. Then $0<q_n\le q_{n+1}<1$ and
\begin{gather}\label{eq:ind1}
u_{n,k}\le q_n^{k^2}\prod_{j=1}^k \big(1-q_n^{2j}\big)^{-1}= \frac{q_n^{k^2}}{\big(q_n^2;q_n^2\big)_k},\qquad n\ge 1, \qquad k\ge 0.
\end{gather}
\end{lem}

\begin{proof} The inequality trivially holds when $k=0$ independent of $n$, because both sides are~1. Since $u_{1,k}=0$ for $k\ge 1$ the inequality holds in particular for $n=1$ and all $k$.
Notice that by assumption
\begin{gather*}
\frac{b_{n-j}}{b_n}=q_{n-j+1}q_{n-j+2}\cdots q_n\le q_n^j.
\end{gather*}
Let us define
\begin{gather*}
\tau_{n,k}:=\frac{q_n^{k^2}}{\big(q_n^2;q_n^2\big)_k}.
\end{gather*}
Since $q_n\le q_{n+1}$ we have $\tau_{n,k}\le \tau_{n+1,k}$. Suppose now that \eqref{eq:ind1} holds for some $n\ge 1$ and all $k$. For $k\ge 1$ we get by \eqref{eq:u} and the induction hypothesis
\begin{gather*}
 u_{n+1,k} \le q_{n}^{2k}u_{n,k}+ q_{n}^{2k-1}u_{n,k-1} \le q_{n}^{2k}\tau_{n,k}+ q_{n}^{2k-1}\tau_{n,k-1}\\
\hphantom{u_{n+1,k}}{} \le q_{n+1}^{2k}\tau_{n+1,k}+q_{n+1}^{2k-1}\tau_{n+1,k-1}=\tau_{n+1,k}.\tag*{\qed}
\end{gather*}\renewcommand{\qed}{}
\end{proof}

\begin{thm}\label{thm:logconvex}Assume $(b_n)$ is $q$-increasing, cf.\ Definition~{\rm \ref{thm:defqinc}}, and let $U_n$ be given by~\eqref{eq:Un}. Then
 \begin{gather*}%\label{eq:logconvexU}
U_n\le \big(q^2;q^2\big)_\infty^{-2} \sum_{k=0}^{\infty}q^{2k^2},\qquad n\ge 0.
\end{gather*}
This holds in particular if $(b_n)$ is log-convex and strictly increasing with $q=b_0/b_1<1$.
\end{thm}

\begin{proof}For $l\ge 1$ we have
\begin{gather}\label{eq:logcv-q}
\frac{b_n}{b_{n+l}}=\frac{b_n}{b_{n+1}}\frac{b_{n+1}}{b_{n+2}}\cdots\frac{b_{n+l-1}}{b_{n+l}}\le q^l,
\end{gather}
 so from equation \eqref{eq:u} we get
\begin{gather}\label{eq:uq}
u_{n+1,k}\le q^{2k} u_{n,k}+q^{2k-1}u_{n,k-1}.
\end{gather}
We claim that
\begin{gather}\label{eq:ind1*}
u_{n,k}\le q^{k^2}\prod_{j=1}^k\big(1-q^{2j}\big)^{-1}=\frac{q^{k^2}}{\big(q^2;q^2\big)_k},
\end{gather}
which is certainly true for $n=0,k\ge 0$ since $u_{0,0}=1$ and $u_{0,k}=0$ for $k\ge 1$. It is also true when $k=0$ independent of $n$ since $u_{n,0}=1$.

Assume that \eqref{eq:ind1*} holds for some $n$ and all $k$. From \eqref{eq:uq} we get for $k\ge 1$
\begin{gather*}
u_{n+1,k}\le \frac{q^{k^2+2k}}{\big(q^2;q^2\big)_k}+\frac{q^{(k-1)^2+2k-1}}{\big(q^2;q^2\big)_{k-1}}=\frac{q^{k^2}}{\big(q^2;q^2\big)_k},
\end{gather*}
which finishes the proof of \eqref{eq:ind1*}. In particular, using the notation from~\eqref{eq:q-infty}
\begin{gather*}
u_{n,k}\le \frac{q^{k^2}}{\big(q^2;q^2\big)_\infty},
\end{gather*}
hence by \eqref{eq:u1}
\begin{gather*}
U_n\le \big(q^2;q^2\big)_\infty^{-2} \sum_{k=0}^{\infty}q^{2k^2}<\infty.\tag*{\qed}
\end{gather*}\renewcommand{\qed}{}
\end{proof}

Theorem~\ref{thm:logconvex} can be slightly extended.

\begin{thm}\label{thm:evtlogconvex} Assume $(b_n)$ is eventually $q$-increasing, cf.\ Definition~{\rm \ref{thm:defqinc}}. Then the sequen\-ce~$(U_n)$ from~\eqref{eq:Un} is bounded.
\end{thm}

\begin{proof} By \eqref{eq:qinc} we see that $(b_n)$ is strictly increasing for $n\ge n_0$ and tends to infinity. For $k\ge 1$ we then have
\begin{gather*}
\max\{b_n\colon 0\le n<n_0+k\}=\max\left(\max\{b_n\colon 0\le n<n_0\}, b_{n_0+k-1}\right)
\end{gather*}
and the latter is $<b_{n_0+k}$ for $k$ sufficiently large.

By increasing $n_0$ if necessary we may therefore assume that
\begin{gather*}
b_{n_0}> \max\{b_n\colon 0\le n<n_0\}.
\end{gather*}

Defining
\begin{gather*}
q_1=\max_{0\le n<n_0} \left(\frac{b_n}{b_{n_0}}\right)^{1/(n_0-n)}, \qquad p=\max(q,q_1),
\end{gather*}
we have $0<p<1$.

Let
\begin{gather*}
\tilde{b}_n:=\begin{cases} b_{n_0}p^{n_0-n}, & 0\le n<n_0,\\
b_n, & n_0\le n.
\end{cases}
\end{gather*}
Then $\big(\tilde{b}_n\big)$ is $p$-increasing and $b_n\le\tilde{b}_n$. By Proposition~\ref{thm:increasofbn} we get $s_{2n}\le\tilde{s}_{2n}$ and for $n >n_0$ we find
\begin{gather*}
U_n=\frac{s_{2n}}{b_0^2b_1^2\cdots b_{n-1}^2}\le \frac{\tilde{b}_0^2\cdots \tilde{b}_{n_0-1}^2}{b_0^2\cdots b_{n_0-1}^2} \tilde{U}_n.
\end{gather*}
This shows the boundedness of $(U_n)$ since $(\tilde{U}_n)$ is bounded by Theorem~\ref{thm:logconvex}.
\end{proof}

\begin{prop}\label{thm:unk-logconcave}Assume $(b_n)$ is strictly increasing and log-concave. Then $u_{n,k}$ from \eqref{eq:u0} satisfy
\begin{gather}\label{eq:unk-logconcave}
u_{n,0}=1, \qquad u_{n,k}\le \prod_{j=1}^k {b_{n-2k+j-1}\over b_{n-j}-b_{n-2k+j-2}},\qquad 1\le k\le {n\over 2}.
\end{gather}
\end{prop}

\begin{proof}We will use induction with respect to $n$. Equation~\eqref{eq:unk-logconcave} is certainly true for $n=1$. Assume the conclusion is satisfied for $n\ge 1$. For $1\le j\le k$, $k\ge 2$ we use log-concavity to obtain
\begin{gather*}
\frac{b_{n-2k+j-1}}{b_{n-j}-b_{n-2k+j-2}}\le \frac{b_{n-2k+j}}{b_{n+1-j}-b_{n-2k+j-1}},
\end{gather*}
hence by the induction hypothesis
\begin{gather*}
u_{n+1,k} = {b_{n-2k}\over b_n}u_{n,k}+ {b_{n-2k+1}\over b_n}u_{n,k-1}\\
\hphantom{u_{n+1,k}}{} \le {b_{n-2k}\over b_n}\prod_{j=1}^k {b_{n-2k+j-1}\over b_{n-j}-b_{n-2k+j-2}}+
{b_{n-2k+1}\over b_n}\prod_{j=1}^{k-1} {b_{n-2k+j+1}\over b_{n-j}-b_{n-2k+j}}\\
\hphantom{u_{n+1,k}}{} \le {b_{n-2k}\over b_n}\prod_{j=1}^k {b_{n-2k+j}\over b_{n+1-j}-b_{n-2k+j-1}}+
{b_{n-2k+1}\over b_n}\prod_{j=2}^{k} {b_{n-2k+j}\over b_{n+1-j}-b_{n-2k+j-1}} \\
\hphantom{u_{n+1,k}}{} = \prod_{j=1}^k {b_{n-2k+j}\over b_{n+1-j}-b_{n-2k+j-1}}\left [{b_{n-2k}\over b_n}+{b_{n-2k+1}\over b_n}{b_n-b_{n-2k}\over b_{n-2k+1}}\right ]\\
\hphantom{u_{n+1,k}}{} = \prod_{j=1}^k {b_{n-2k+j}\over b_{n+1-j}-b_{n-2k+j-1}}.
\end{gather*}
For $k=1$ we similarly have
\begin{gather*}
u_{n+1,1} = {b_{n-2}\over b_n}u_{n,1}+{b_{n-1}\over b_n}\le {b_{n-2}\over b_n}{b_{n-2}\over b_{n-1}-b_{n-3}}+{b_{n-1}\over b_n}\\
\hphantom{u_{n+1,1}}{} \le {b_{n-2}\over b_n}{b_{n-1}\over b_{n}-b_{n-2}}+{b_{n-1}\over b_n}={b_{n-1}\over b_{n}-b_{n-2}}.\tag*{\qed}
\end{gather*}\renewcommand{\qed}{}
\end{proof}

\begin{thm}\label{thm:thm2}Assume $(b_n)$ satisfies
\begin{gather}\label{eq:b_{n-1}/b_n}
{b_{n-1}\over b_{n}}\le {\rm e}^{-\beta/n}, \qquad n\ge 1,\qquad \beta>0.
\end{gather}
The quantity $U_n$ from \eqref{eq:Un} satisfies
\begin{gather}\label{eq:rootUn}
\limsup_{n\to\infty} U_n^{1/n}\le {\rm e}^{2u(\beta)},
\end{gather}
where
\begin{eqnarray}\label{eq:mbeta}
u(\beta):=\max_{2^{-1/\beta}\le x\le 1} (x+1) \int_x^1 \log\left(\frac{y^\beta}{1-y^\beta}\right)\frac{{\rm d}y}{(1+y)^2}>0.
\end{eqnarray}
The function $u(\beta)$ is decreasing with $\lim\limits_{\beta\to 0}u(\beta)=\infty$, $\lim\limits_{\beta\to\infty}u(\beta)=0$.
\end{thm}

\begin{proof} By \eqref{eq:b_{n-1}/b_n} we get
\begin{gather*}
b_n\ge b_0{\rm e}^{\beta(1+1/2+\dots +1/n)},
\end{gather*}
showing that $(b_n)$ is strictly increasing to infinity.

Defining{\samepage
\begin{gather*}
\tilde{b}_0:=b_0,\qquad \tilde{b}_n:=b_0 {\rm e}^{\beta(1+1/2+\dots + 1/n)},\qquad n\ge 1,
\end{gather*}
then $\big(\tilde{b}_n\big)$ is strictly increasing and log-concave.}

For $\widehat{b}_n=(n+1)^\beta$ we then have
\begin{gather}\label{eq:3bn}
\frac{b_{n-1}}{b_n}\le \frac{\tilde{b}_{n-1}}{\tilde{b}_n}\le \frac{\widehat{b}_{n-1}}{\widehat{b}_n},
\end{gather}
so by Proposition~\ref{thm:comparison}
\begin{gather*}
U_n\le \tilde{U}_n\le \widehat{U}_n.
\end{gather*}
We claim that
\begin{gather}\label{eq:lcavsimp}
\prod_{j=1}^k {\tilde{b}_{n-2k+j-1}\over \tilde{b}_{n-j}-\tilde{b}_{n-2k+j-2}}\le \prod_{j=1}^k {\widehat{b}_{n-2k+j-1}\over\widehat{b}_{n-j}-\widehat{b}_{n-2k+j-2}}=:\sigma_{n,k},
\end{gather}
which follows from the second inequality in \eqref{eq:3bn} because of
\begin{gather*}
{\tilde{b}_{n-2k+j-1}\over \tilde{b}_{n-j}-\tilde{b}_{n-2k+j-2}}=\frac{\tilde{b}_{n-2k+j-1}/\tilde{b}_{n-j}}{1-\tilde{b}_{n-2k+j-2}/\tilde{b}_{n-j}}
\end{gather*}
and a similar expression for $\widehat{b}_n$. Let $k_n$ be the smallest index so that
\begin{gather*}
\sigma_{n,k_n}=\max\{\sigma_{n,k}\colon 1\le k\le n/2\}.
\end{gather*}
We have $n=k_n(\gamma_n+2),$ where $\gamma_n\ge 0$.
By Proposition~\ref{thm:unk-logconcave} and the inequality \eqref{eq:lcavsimp} we have $\tilde{u}_{n,k}\le \sigma_{n,k}$, hence
\begin{gather}
 \limsup_n U_n^{1/n}\le \limsup_n \tilde{U}_n^{1/n}\le \limsup_n\left( \sum_{k=0}^{[n/2]}\sigma_{n,k}^2\right )^{1/n}\le \limsup_n\big(1+[n/2]\sigma_{n,k_n}^2\big)^{1/n}\nonumber\\
\hphantom{\limsup_n U_n^{1/n}}{} \le \limsup_n (1+[n/2])^{1/n}\sigma_{n,k_n}^{2/n} \le \limsup_n \sigma_{n,k_n}^{2/n},\label{eq:lcavsimp1}
\end{gather}
where we have used that $\sigma_{n,0}=1$ and
\begin{gather*}
\sigma_{n,k_n}\ge \sigma_{n,1}=\frac{(n-1)^\beta}{n^\beta-(n-2)^\beta}\to\infty
\end{gather*}
for $n\to\infty$. By the right-hand side of \eqref{eq:lcavsimp} we get
\begin{gather*}
\limsup_n\sigma_{n,k_n}^{1/n}= \limsup_n \left (\prod_{j=0}^{k_n-1} {\widehat{b}_{n-2k_n+j}\over \widehat{b}_{n-j-1}-\widehat{b}_{n-2k_n+j-1}}\right )^{1/n}.
\end{gather*}

Furthermore
\begin{gather}\label{eq:hpa1}
\log \left(\prod_{j=0}^{k_n-1} {\widehat{b}_{n-2k_n+j}\over \widehat{b}_{n-j-1}-\widehat{b}_{n-2k_n+j-1}}\right)^{1/n} =\frac{1}{(\gamma_n+2)k_n}\sum_{j=0}^{k_n-1} E_n(\beta,j),
\end{gather}
where
\begin{gather}\label{eq:hpa2}
E_n(\beta,j)= \beta\log (\gamma_n+{j/k_n}+1/k_n) -\log \big[(\gamma_n+2-{j/k_n} )^\beta - (\gamma_n+{j/ k_n})^\beta \big].
\end{gather}

Let $\gamma$ be an accumulation point of $(\gamma_n)$.

We consider first the case, where $\gamma=\infty$ is the only accumulation point, so $\gamma_n\to\infty$. We have
\begin{gather}\label{eq:hp1}
\log(\gamma_n+{j/ k_n}+1/k_n)\le \log(\gamma_n+1).
\end{gather}

The inequality $(0<y<x)$
\begin{gather*}
{x^\beta-y^\beta\over x-y}\ge \begin{cases}
 \beta y^{\beta-1}, & \beta\ge 1,\\
 \beta x^{\beta-1}, & 0<\beta<1,
 \end{cases}
\end{gather*}
implies
\begin{gather}
 (\gamma_n+2-{j/ k_n} )^\beta- (\gamma_n+{j/ k_n})^\beta \ge \begin{cases}
2\beta (1-j/k_n)(\gamma_n+{j/ k_n} )^{\beta-1},& \beta\ge 1, \\
2\beta (1-j/k_n)(\gamma_n+2-{j/ k_n} )^{\beta-1}, & 0<\beta<1
\end{cases} \nonumber\\
\hphantom{(\gamma_n+2-{j/ k_n} )^\beta- (\gamma_n+{j/ k_n})^\beta}{} \ge \begin{cases}
 2\beta (1-j/k_n)\gamma_n^{\beta-1}, &\beta \ge 1,\\
 2\beta (1-j/k_n)(\gamma_n+2)^{\beta-1}, & 0<\beta<1.
\end{cases}\label{eq:hp2}
\end{gather}
We have
\begin{gather}\label{eq:hp3}
-\frac{1}{k_n}\sum_{j=0}^{k_n-1}\log (1-j/k_n)\le -\int_0^1\log x{\rm d}x=1.
\end{gather}
By \eqref{eq:hp1}, \eqref{eq:hp2} and \eqref{eq:hp3} we get that the right-hand side of \eqref{eq:hpa1} is bounded above by $C\log(\gamma_n+2)/(\gamma_n+2)$ for suitable $C>0$, hence tends to 0. In this case~\eqref{eq:lcavsimp1} gives
\begin{gather*}
\limsup_n U_n^{1/n}\le 1,
\end{gather*}
and \eqref{eq:rootUn} holds.

Assume now that $\gamma<\infty$. Then, in view of \eqref{eq:hpa1} and \eqref{eq:hpa2} the quantity
\begin{gather*}
I={1\over \gamma+2} \int_0^1 \big\{\beta\log (\gamma+x) -\log \big[(\gamma+2-x)^\beta-(\gamma+x)^\beta\big]\big\}{\rm d}x
\end{gather*}
is an accumulation point of \eqref{eq:hpa1}. The substitution
\begin{gather*}
 y={\gamma+x\over \gamma+2-x}
\end{gather*}
 leads to
\begin{gather*}
I ={2\gamma+2\over \gamma+2}\int_{\gamma\over \gamma+2}^1\log \left({y^\beta\over 1-y^\beta}\right) {{\rm d}y\over (1+y)^2}.
\end{gather*}
Set $\sigma=\gamma /(\gamma+2)$. Then
\begin{gather*}
I =(\sigma+1)\int_\sigma^1\log \left({y^\beta\over 1-y^\beta}\right) {{\rm d}y\over (1+y)^2},
\end{gather*}
hence
\begin{gather*}
I \le \max_{0\le x\le 1} \left \{(x+1)\int_x^1\log \left({y^\beta\over 1-y^\beta}\right) {{\rm d}y\over (1+y)^2}\right \}.
\end{gather*}
Clearly the maximum is attained for $x\ge 2^{-1/\beta}$, as in this case the integrated function is nonnegative. Now~\eqref{eq:rootUn} follows.

The last assertion is easy to prove.
\end{proof}

\begin{thm}\label{thm:thm2a} Assume $(b_n)$ satisfies
\begin{gather}\label{eq:hpb1}
\frac{b_{n-1}}{b_n}\le {\rm e}^{-f(n)},\qquad n\ge n_0\ge 1,
\end{gather}
where $f(n)>0$ for $n\ge n_0$ and
\begin{gather*}
\alpha:=\liminf_{n\to\infty} nf(n)>0.
\end{gather*}
The quantity $U_n$ from \eqref{eq:Un} satisfies
\begin{gather}\label{eq:rootUn1}
\limsup_n U_n^{1/n}\le {\rm e}^{2u(\alpha)},
\end{gather}
where $u(\alpha)$ is defined in \eqref{eq:mbeta} and $u(\infty):=0$.
\end{thm}

\begin{proof}Let $0<\beta<\alpha$ be arbitrary. There exists $N\ge n_0$ such that
\begin{gather}\label{eq:hpb2}
nf(n)\ge \beta,\qquad n\ge N.
\end{gather}
Combining \eqref{eq:hpb1} and \eqref{eq:hpb2}, we see that $(b_n)$ is eventually strictly increasing to infinity. By replacing $N$ by a larger integer if necessary, we may assume that
\begin{gather*}
\max\{b_n \colon 0\le n\le N-1\}<b_{N},
\end{gather*}
see the proof of Theorem~\ref{thm:evtlogconvex}. Defining
\begin{gather*}
\tilde{b}_n:=\begin{cases} b_n,& 0\le n\le N, \\
\displaystyle b_N\exp\left[\beta\left(\frac{1}{N+1}+\frac{1}{N+2}+\dots +\frac{1}{n}\right)\right],& n>N,
\end{cases}
\end{gather*}
we have
\begin{gather*}
\frac{b_{n-1}}{b_n}\le \frac{\tilde{b}_{n-1}}{\tilde{b}_n},\qquad n\ge 1,
\end{gather*}
hence $U_n\le \tilde{U}_n$ by Proposition~\ref{thm:comparison}.

We see that $\tilde{b}_n\le \tilde{b}_n^{(N)}:=\tilde{b}_{n+N}$ and $\tilde{b}_n^{1/n}\to 1$, so by Proposition~\ref{thm:shiftU} we get
\begin{gather}\label{eq:hpc1}
\limsup_n U_n^{1/n}\le \limsup_n \tilde{U}_n^{1/n}=\limsup_n\big(\tilde{U}_n^{(N)}\big)^{1/n}.
\end{gather}
Define now
\begin{gather*}
\widehat{b}_0 = b_N\exp\left[-\beta\left(1+\frac12+\dots +\frac{1}{N}\right)\right],\\
\widehat{b}_n = \widehat{b}_0\exp\left[\beta\left(1+\frac12+\ldots +\frac{1}{n}\right)\right],\qquad n\ge 1.
\end{gather*}
Then $\big(\widehat{b}_n\big)$ is strictly increasing and satisfies $\widehat{b}_{n-1}/\widehat{b}_{n}=\exp(-\beta/n)$, $n\ge 1$. Furthermore, the conditions of Proposition~\ref{thm:shiftU} are satisfied so
\begin{gather}\label{eq:hpc2}
\limsup_n \widehat{U}_n^{1/n}=\limsup_n\big(\widehat{U}_n^{(N)}\big)^{1/n}.
\end{gather}
Notice that $\widehat{b}_n^{(N)}=\tilde{b}_n^{(N)}$, so by \eqref{eq:hpc1}, \eqref{eq:hpc2} and Theorem~\ref{thm:thm2}
\begin{gather*}
\limsup_n U_n^{1/n}\le \limsup_n \widehat{U}_n^{1/n}\le {\rm e}^{2u(\beta)}.
\end{gather*}
Since $\beta<\alpha$ is arbitrary, we finally get \eqref{eq:rootUn1} by continuity.
\end{proof}

\begin{rem}\label{rem1}
The assumptions of Theorem~\ref{thm:thm2a} are satisfied for the following sequences, where it is assumed that $n$ is so large that the expressions are defined and positive:
\begin{align*}&b_n={\rm e}^{n^\gamma}, && 0<\gamma <1, && f(n)=\gamma n^{\gamma-1},&&\alpha=\infty,\\
& b_n={\rm e}^{(\log n)^\gamma}, &&\gamma >1,&& f(n)={\gamma\over n} \log ^{\gamma-1}n,&&\alpha=\infty,\\
& b_n={\rm e}^{\log n (\log\log n)^\gamma},\ &&\gamma >1, &&f(n)={\log \log n+\gamma\over n}(\log\log n)^{\gamma-1},&&\alpha=\infty,\\
& b_n=n^\gamma ,&& \gamma>0,&& f(n)={\gamma\over n},&& \alpha=\gamma .
\end{align*}
\end{rem}

\begin{rem}\label{rem3}Assume $(b_n)$ is increasing. Then by (\ref{eq:u'}) we get
\begin{gather*}
u_{n+1,k}\le u_{n,k}+u_{n,k-1},\qquad k\ge 0.
\end{gather*}
This implies, by an induction argument, that
\begin{gather*}
u_{n,k}\le {n\choose k}.
\end{gather*}
Thus by (\ref{eq:u1}) we get

\begin{gather}\label{eq:4^n}
U_n \le \sum_{k=0}^{[n/2]}{n\choose k}^2\le {2n\choose n}\le 4^n,
\end{gather}
hence
\begin{gather*}
 \limsup_n U_n^{1/n}\le 4.
\end{gather*}
This estimate is weaker than the one obtained in Theorem~\ref{thm:thm2a} for $\alpha >0.30873$. In fact, $\exp(2u(\alpha))\le 4$ for $u(\alpha)\le \log 2$ and $u(0.30872)>\log 2$, $u(0.30783)< \log 2$.

However, for bounded sequences $(b_n)$ the estimate \eqref{eq:4^n} cannot be improved, see Appendix~\ref{appendixA}.
\end{rem}

\section[Estimates of $c_n$ via $V_n$]{Estimates of $\boldsymbol{c_n}$ via $\boldsymbol{V_n}$}\label{section4}

In all of this section we consider symmetric moment problems given by a positive sequence~$(b_n)$, and it is tacitly assumed that the moment problem is indeterminate in order to be able to define~$c_n$ from~\eqref{eq:c_k}, hence $\sum 1/b_n<\infty$ by a classical result due to Carleman, cf.~\cite[p.~24]{Ak}.

For $k,l\ge 0$ define
\begin{gather}\label{eq:vkl}
v_{k,l}:= (-1)^l b_0b_1\cdots b_{k-1} b_{k,k+2l},
\end{gather}
where $b_{k,k+2l}$ is the coefficient to $x^k$ in $P_{k+2l}$, so $(-1)^lb_{k,k+2l}>0$ and hence $v_{k,l}>0$.

Observe that $v_{k,0}=1$ and since $b_{k,j}=0$ if $j<k$ it makes sense to define $v_{k,l}=0$ if $l<0$.

From \eqref{eq:P2n(0)} we know
\begin{gather}\label{eq:v0l}
P_{2l}(0)=b_{0,2l}=(-1)^lv_{0,l}=(-1)^l\prod_{j=1}^l \frac{b_{2j-2}}{b_{2j-1}}.
\end{gather}

Notice that by \eqref{eq:c_k} and \eqref{eq:Vn}
\begin{gather}\label{eq:V_n1}
V_n=b_0b_1\cdots b_{n-1} c_n=\left(\sum_{l=0}^\infty v_{n,l}^2\right)^{1/2}.
\end{gather}

\begin{lem}\label{thm:vkl}
The coefficients $v_{k,l}$ satisfy
\begin{gather}\label{eq:vkl1} v_{k+1,l}={b_k\over b_{k+2l}}v_{k,l}+ {b_{k+2l-1}\over b_{k+2l}}v_{k+1,l-1},\qquad k,l\ge 0.
\end{gather}
\end{lem}

\begin{proof}We have
\begin{gather*}
x P_{k+2l}=b_{k+2l}P_{k+2l+1}+b_{k+2l-1}P_{k+2l-1}.
\end{gather*}
By equating the coefficients of $x^{k+1}$ we obtain
\begin{gather*}
b_{k,k+2l}=b_{k+2l}b_{k+1,k+2l+1}+ b_{k+2l-1}b_{k+1,k+2l-1}.
\end{gather*}
Multiplying the last identity by $(-1)^lb_0b_1\cdots b_{k-1}$ leads to
\begin{gather*}
 v_{k,l}={b_{k+2l}\over b_k}v_{k+1,l}- { b_{k+2l-1}\over b_k}v_{k+1,l-1}.
\end{gather*}
 This gives (\ref{eq:vkl1}).
\end{proof}

The following result follows easily from \eqref{eq:v0l} and \eqref{eq:vkl1}.

\begin{prop}\label{thm:compV} Assume $(b_n)$ and $\big(\tilde{b}_n\big)$ are positive sequences satisfying
\begin{gather*}
{b_{n-1}\over b_{n}} \le {\tilde{b}_{n-1}\over \tilde{b}_{n}} ,\qquad n\ge 1.
\end{gather*}
 Then $v_{k,l} \le \tilde{v}_{k,l} $ and $V_k \le \tilde{V}_k$.
\end{prop}

For $n_0\in\mathbb N$ the shifted sequence $b_n^{(n_0)}:=b_{n+n_0}$ also leads to an indeterminate moment problem, cf.~\cite[p.~28]{Ak}, and we shall consider the associated quantities $v_{k,l}^{(n_0)}$, $V_k^{(n_0)}$.

\begin{prop}\label{thm:shiftV1} The expressions above satisfy
\begin{gather}\label{eq:shiftV1}
v_{k,l}^{(1)}\le v_{k+1,l},\qquad V_k^{(1)}\le V_{k+1},\qquad k,l\ge 0,\\
\label{eq:shiftV2}
v_{k+1,l}\le L v_{k,l}^{(1)},\qquad V_{k+1}\le L V_k^{(1)},\qquad k,l\ge 0,
\end{gather}
where $L:=\sum\limits_{n=0}^\infty P_n^2(0)$. For $n_0\in\mathbb N$ we similarly have
\begin{gather}\label{eq:shiftV1a}
v_{k,l}^{(n_0)}\le v_{k+n_0,l}\le L(n_0) v_{k,l}^{(n_0)},\qquad V_k^{(n_0)}\le V_{k+n_0}\le L(n_0) V_k^{(n_0)},
\end{gather}
for a suitable constant $L(n_0)>0$. Furthermore,
\begin{gather} \label{eq:shiftV3}
\limsup_{k\to\infty} V_k^{1/k}=\limsup_{k\to\infty} \big(V_k^{(n_0)}\big)^{1/k}.
\end{gather}
\end{prop}

\begin{proof} From \eqref{eq:vkl1} we get
\begin{gather}\label{eq:shiftvkl1}
 v_{k+1,l}^{(1)}=\frac{b_{k+1}}{b_{k+2l+1}}v_{k,l}^{(1)}+ \frac{b_{k+2l}}{b_{k+2l+1}}v_{k+1,l-1}^{(1)},\qquad k,l\ge 0,\\
 \label{eq:vkl1+}
v_{k+2,l}=\frac{b_{k+1}}{b_{k+2l+1}}v_{k+1,l}+ \frac{b_{k+2l}}{b_{k+2l+1}}v_{k+2,l-1},\qquad k,l\ge 0.
\end{gather}

On the other hand
\begin{gather}\label{eq:v1l}
v_{1,l}=\frac{b_0}{b_{2l}} v_{0,l}+\frac{b_{2l-1}}{ b_{2l}}v_{1,l-1}\ge \frac{b_{2l-1}}{ b_{2l}}v_{1,l-1}.
\end{gather}
Iterating this we get using $v_{1,0}=1$
\begin{gather}\label{eq:v0l(1)}
v_{1,l}\ge \prod_{j=1}^l \frac{b_{2j-1}}{b_{2j}}=v_{0,l}^{(1)},
\end{gather}
where the last equality follows from \eqref{eq:v0l}. This proves the left inequality in~\eqref{eq:shiftV1} for $k=0$ and all $l$, and it is clear for $l=0$ and all~$k$. In particular it is true for $k+l\le 1$. Assume that it holds for $k+l\le n$ for some $n\ge 1$. To prove it for $k+l=n+1$, it suffices to prove $v_{k+1,l}^{(1)}\le v_{k+2,l}$ when $l\ge 1$, but this follows immediately by~\eqref{eq:shiftvkl1} and~\eqref{eq:vkl1+} and the induction hypothesis. The right inequality in~\eqref{eq:shiftV1} is an immediate consequence of the left inequality.

Defining $t_l=v_{1,l}/v_{0,l}^{(1)}$, $l\ge 0$ we get by \eqref{eq:v0l(1)}
\begin{gather*}
t_l=v_{1,l}\prod_{j=1}^l \frac{b_{2j}}{b_{2j-1}},\qquad l\ge 1,
\end{gather*}
so by \eqref{eq:v1l} and \eqref{eq:v0l}
\begin{gather*}
t_l=\left(\frac{b_0}{b_{2l}}v_{0,l}+\frac{b_{2l-1}}{b_{2l}}v_{1,l-1}\right)\prod_{j=1}^l \frac{b_{2j}}{b_{2j-1}}=\frac{b_0}{b_{2l}}\prod_{j=1}^l\frac{b_{2j-2}b_{2j}}{b_{2j-1}^2} + t_{l-1}.
\end{gather*}
We then get for $l\ge 1$
\begin{gather*}
t_l=t_0+\sum_{m=1}^l \frac{b_0}{b_{2m}}\prod_{j=1}^m \frac{b_{2j-2}b_{2j}}{b_{2j-1}^2} =1+\sum_{m=1}^l P_{2m}^2(0)\le \sum_{n=0}^\infty P_n^2(0)=L,
\end{gather*}
hence $v_{1,l}\le L v_{0,l}^{(1)}$ for all $l$. By induction as above we then get the left side of~\eqref{eq:shiftV2}, and the right side follows.

The inequalities \eqref{eq:shiftV1a} easily follow by iteration with a suitable constant $L(n_0)>0$ independent of~$k$, and \eqref{eq:shiftV3} follows.
\end{proof}

\begin{thm}\label{thm:V_nlogcv} Assume $(b_n)$ is $q$-increasing, cf.\ Definition~{\rm \ref{thm:defqinc}}. The quantity $V_n$ from \eqref{eq:Vn} satisfies
\begin{gather}\label{eq:V_nlogcv}
V_n\le\frac{1}{\big(q^2;q^2\big)_\infty\sqrt{1-q^2}}.
\end{gather}
This holds in particular if $(b_n)$ is log-convex and strictly increasing with $q=b_0/b_1<1$.
\end{thm}

\begin{proof} Note that the moment problem is indeterminate by Proposition~\ref{thm:qinc-indet}. From~\eqref{eq:vkl1} and \eqref{eq:logcv-q} we obtain
\begin{gather}\label{eq:vkl1q}
v_{k+1,l}\le q^{2l} v_{k,l}+qv_{k+1,l-1}.
\end{gather}

We claim that
\begin{gather}\label{eq:vklq}
v_{k,l}\le q^l\prod_{j=1}^l \big(1-q^{2j}\big)^{-1}=\frac{q^l}{\big(q^2;q^2\big)_l},\qquad k,l\ge 0,
\end{gather}
which is certainly true for $l=0$ since $v_{k,0}=1$ and for $k=0$ since by \eqref{eq:v0l}
\begin{gather*}
v_{0,l}\le q^l\le \frac{q^l}{\big(q^2;q^2\big)_l}.
\end{gather*}

 We shall proceed by induction in $k+l$, and as already observed \eqref{eq:vklq} holds when $k+l\le 1$.

Suppose \eqref{eq:vklq} holds for $k+l\le n$ for some $n\ge 1$ and we shall prove it for $k+l=n+1$. From \eqref{eq:vkl1q} and the induction hypothesis we get for $l\ge 1$
\begin{gather*}
v_{k+1,l} \le q^{2l}v_{k,l}+q\, v_{k+1,l-1} \le \frac{q^{3l}}{\big(q^2;q^2\big)_{l}}+\frac{q^{l}}{\big(q^2;q^2\big)_{l-1}}=\frac{q^{l}}{\big(q^2;q^2\big)_{l}}.
\end{gather*}

From \eqref{eq:vklq} we get
\begin{gather*}
v_{k,l}\le \frac{q^l}{\big(q^2;q^2\big)_\infty}
\end{gather*}
and finally
\begin{gather*}
V_n^2=\sum_{l=0}^\infty v_{n,l}^2\le \big(q^2;q^2\big)_\infty^{-2}\sum_{l=0}^\infty q^{2l},
\end{gather*}
leading to \eqref{eq:V_nlogcv}.
\end{proof}

The boundedness of $(V_n)$ also holds in a slightly more general case.

\begin{thm}\label{thm:Vevtlogconv} Assume $(b_n)$ is eventually $q$-increasing, cf.\ Definition~{\rm \ref{thm:defqinc}}. The sequence $(V_n)$ from \eqref{eq:Vn} is bounded.
\end{thm}

\begin{proof} Note that the moment problem is indeterminate by Proposition~\ref{thm:qinc-indet}. We assume that \eqref{eq:qinc} holds. Then the sequence $\big(b_n^{(n_0)}\big)$ is $q$-increasing, so $\big(V_n^{(n_0)}\big)$ is bounded by Theorem~\ref{thm:V_nlogcv}. The boundedness of $(V_n)$ follows from \eqref{eq:shiftV1a}.
\end{proof}

Combining Theorem~\ref{thm:evtlogconvex} and Theorem~\ref{thm:Vevtlogconv} we get:

\begin{cor}\label{thm:VnsqrtUn} Assume $(b_n)$ is eventually $q$-increasing. Then $c_n\sqrt{s_{2n}}=V_n\sqrt{U_n}$ is a bounded sequence.
\end{cor}

We now turn to the log-concave case.

\begin{prop}\label{thm:aV_k} Assume $(b_n)$ is log-concave and strictly increasing. Then
\begin{gather*}
V_k\le V_{k+1}, \qquad c_k\le b_kc_{k+1}, \qquad k\ge 0.
\end{gather*}
\end{prop}

 \begin{proof} We start by showing that
\begin{gather}\label{eq:v01}
v_{0,1} < v_{1,1}.
\end{gather}
Indeed, as
\begin{gather*}
P_2(x)={1\over b_0b_1}x^2-{b_0\over b_1},\qquad P_3(x)={1\over b_0b_1b_2}x^3-{b_0^2+b_1^2\over b_0b_1b_2}x,
\end{gather*}
 we have{\samepage
 \begin{gather*}
v_{0,1}= {b_0\over b_1},\qquad v_{1,1}={b_0^2+b_1^2\over b_1 b_2},
\end{gather*}
and using $b_0/b_1\le b_1/b_2$ we get \eqref{eq:v01}.}

 We are going to show that
 \begin{gather}\label{eq:avk}
v_{k,l}\le v_{k+1,l},
 \end{gather}
 by induction with respect to $k+l$. It is certainly true for $l=0$ and for $k=0$, $l=1$ by \eqref{eq:v01}, and it therefore holds for $k+l\le 1$.

We first show \eqref{eq:avk} for $k=0$, i.e.,
\begin{gather*}
v_{0,l}\le v_{1,l},\qquad l\ge 0,
\end{gather*}
and we use induction in $l$. Assuming it holds for $l-1\ge 0$ we get by \eqref{eq:vkl1}
\begin{gather*}
v_{1,l} = \frac{b_0}{b_{2l}}v_{0,l}+\frac{b_{2l-1}}{b_{2l}}v_{1,l-1} \ge \frac{b_0}{b_{2l}}v_{0,l}+\frac{b_{2l-1}}{b_{2l}}v_{0,l-1}
= v_{0,l}\left(\frac{b_0}{b_{2l}}+\frac{b_{2l-1}}{b_{2l}}\frac{b_{2l-1}}{b_{2l-2}} \right)> v_{0,l},
\end{gather*}
where we have used \eqref{eq:v0l} and $b_{2l-1}^2\ge b_{2l}b_{2l-2}$.

Assume now that \eqref{eq:avk} holds for $k+l\le n$ for some $n\ge 1$. By log-concavity
\begin{gather*}
\frac{b_n}{b_{n+l}}\le\frac{b_{n+1}}{b_{n+l+1}},\qquad n,l\ge 0,
\end{gather*}
and therefore for $l\ge 1$
\begin{gather*}
 v_{k+1,l} = \frac{b_{k}}{b_{k+2l}} v_{k,l}+ \frac{b_{k+2l-1}}{b_{k+2l}} v_{k+1,l-1} \le \frac{b_{k+1}}{b_{k+2l+1}}v_{k+1,l}+ \frac{b_{k+2l}}{b_{k+2l+1}}v_{k+2,l-1} =v_{k+2,l}.
\end{gather*}
 Finally
\begin{gather*}
V_k=\left(\sum_{l=0}^\infty v_{k,l}^2\right)^{1/2} \le \left(\sum_{l=0}^\infty v_{k+1,l}^2\right)^{1/2}=V_{k+1},
\end{gather*}
hence $c_k\le b_kc_{k+1}$.
 \end{proof}

\begin{thm}\label{thm:vkl-logcc} Assume $(b_n)$ is log-concave and strictly increasing. Set $b_{-1}=0$. The quanti\-ties~$v_{k,l}$ from \eqref{eq:vkl} satisfy
\begin{gather*}
v_{k,l}\le \prod_{j=1}^l {b_{k+2j-2}\over b_{k+2j-1}-b_{k-1}},\qquad k, l\ge 0.
\end{gather*}
 \end{thm}

 \begin{proof} Define
 \begin{gather*}
w_{k,l}:= \prod_{j=1}^l {b_{k+2j-2}\over b_{k+2j-1}-b_{k-1}},\qquad k,l\ge 0,
\end{gather*}
where as usual empty products are equal to 1. We are going to show that
 \begin{gather}\label{eq:v-hatv}
v_{k,l}\le w_{k,l}.
\end{gather}
 Clearly $v_{k,0}= w_{k,0}=1$. Moreover, by \eqref{eq:v0l}
\begin{gather*}
v_{0,l}=w_{0,l}=\prod_{j=1}^l {b_{2j-2}\over b_{2j-1}},\qquad l\ge 1.
\end{gather*}
Therefore \eqref{eq:v-hatv} is valid for $l=0$ and all $k$ and for $k=0$ and all $l$. Observe that by log-concavity we have
 \begin{gather}\label{eq:inc-k}
 w_{k,l}\le w_{k+1,l}.
\end{gather}
We also have
\begin{gather}\label{eq:vhat-l}
{w_{k+1,l-1}\over w_{k+1,l}}= { b_{k+2l}-b_{k}\over b_{k+2l-1}}.
\end{gather}
Assume \eqref{eq:v-hatv} is valid for $k+l\le n, n\ge 1$ and let us prove it for $k+l=n+1$. We may assume $k,l\ge 1$ and write $k=k'+1$. By the induction hypothesis, in view of \eqref{eq:vkl1}, \eqref{eq:inc-k} and \eqref{eq:vhat-l}, we get
 \begin{gather*}
 v_{k,l} = v_{k'+1,l}\le {b_{k'}\over b_{k'+2l}}w_{k',l}+{b_{k'+2l-1}\over b_{k'+2l}}w_{k'+1,l-1}\\
\hphantom{v_{k,l}}{} \le {b_{k'}\over b_{k'+2l}}w_{k'+1,l}+{b_{k'+2l-1}\over b_{k'+2l}}w_{k'+1,l-1}=w_{k'+1,l}=w_{k,l}.\tag*{\qed}
 \end{gather*}\renewcommand{\qed}{}
 \end{proof}

The previous theorem can be extended to the eventually log-concave case.

\begin{thm}\label{thm:vkl-evtlogcc} Assume $(b_n)$ is log-concave and strictly increasing for $n\ge n_0$. Then there exists $L(n_0)\ge 1$ such that the quantities $v_{k,l}$ from \eqref{eq:vkl} satisfy
\begin{gather*}
v_{k,l}\le L(n_0)\prod_{j=1}^l\frac{b_{k+2j-2}}{b_{k+2j-1}-b_{k-1}},\qquad k\ge n_0, \qquad l\ge 0.
\end{gather*}
\end{thm}

\begin{proof} The sequence $\big(b_n^{(n_0)}\big)$ is log-concave and strictly increasing, so by Theorem~\ref{thm:vkl-logcc}
\begin{gather*}
v_{k,l}^{(n_0)}\le \prod_{j=1}^l {b_{k+n_0+2j-2}\over b_{k+n_0+2j-1}-b_{k+n_0-1}},\qquad k, l\ge 0,
\end{gather*}
and since $v_{k+n_0,l}\le L(n_0)v_{k,l}^{(n_0)}$ by \eqref{eq:shiftV1a}, the assertion follows.
\end{proof}

\begin{thm}\label{thm:genlogcc} Assume $(b_n)$ satisfies
\begin{gather*}
{b_{n-1}\over b_n}\le {\rm e}^{-\beta/n},\qquad n\ge 1, \qquad \beta>1.
\end{gather*}
The quantities $V_n$ from \eqref{eq:Vn} satisfy
\begin{gather}\label{eq:Vn-logcc1}
\limsup_{n} V_n^{1/n}\le {\rm e}^{v(\beta)/2},
\end{gather}
where
\begin{gather}\label{eq:taubeta}
v(\beta):=\int_1^\infty\log\big(\big[1-x^{-\beta}\big]^{-1}\big){\rm d}x.
\end{gather}
The function $v(\beta)$ is positive and decreasing with $\lim\limits_{\beta\to1}v(\beta)=\infty$ and $\lim\limits_{\beta\to\infty}v(\beta)=0$.
\end{thm}

\begin{proof} Note that the moment problem is indeterminate by Proposition~\ref{thm:qinc-indet}.
Defining $\big(\tilde{b}_n\big)$ by
\begin{gather*}
\tilde{b}_0=b_0,\qquad \tilde{b}_n=b_0 {\rm e}^{\beta(1+1/2+\dots + 1/n)},\qquad n\ge 1,
\end{gather*}
then $\tilde{b}_{n-1}/\tilde{b}_n={\rm e}^{-\beta/n}$, so $\big(\tilde{b}_n\big)$ is strictly increasing, log-concave and
\begin{gather*}
\frac{b_{n-1}}{b_n}\le \frac{\tilde{b}_{n-1}}{\tilde{b}_n},
\end{gather*}
hence $V_n\le \tilde{V}_n$ by Proposition~\ref{thm:compV}.

By Theorem~\ref{thm:vkl-logcc}
\begin{gather*}
\tilde{v}_{k,l} \le \prod_{j=1}^l {\tilde{b}_{k+2j-2}\over \tilde{b}_{k+2j-1}-\tilde{b}_{k-1}}=
\prod_{j=1}^l \frac{\tilde{b}_{k+2j-2}/\tilde{b}_{k+2j-1}}{1-\tilde{b}_{k-1}/\tilde{b}_{k+2j-1}} \\
\hphantom{\tilde{v}_{k,l}}{} =
\prod_{j=1}^l {{\rm e}^{-{\beta\over k+2j-1}}\over 1-{\rm e}^{-\beta\sum\limits_{i=k}^{k+2j-1}{1\over i}}} =
{\rm e}^{-\beta \sum\limits_{j=1}^l {1\over k+2j-1}}
\prod_{j=1}^l \bigg( 1-{\rm e}^{-\beta\sum\limits_{i=k}^{k+2j-1}{1\over i}}\bigg )^{-1}.
\end{gather*}
We need the estimates
\begin{gather}\label{eq:hpd1}
\sum_{i=k}^{k+2j-1}{1\over i}\ge \int_k^{k+2j}{1\over x}{\rm d}x=\log(k+2j)-\log k
\end{gather}
and
\begin{gather*}
 \sum_{j=1}^l {1\over k+2j-1} = {1\over 2} \sum_{j=1}^l {1\over j+{k-1\over 2}}\ge
 {1\over 2} \int_{k+1\over 2}^{l+{k+1\over 2}}{1\over x}{\rm d}x = {1\over 2}[\log(k+2l+1)-\log(k+1)].
\end{gather*}
Thus, for $k,l\ge 1$,
\begin{gather*}
\tilde{v}_{k,l} \le \left ( {k+1\over k+2l+1}\right )^{\beta/2} \prod_{j=1}^l \left [1-\left ({k\over k+2j}\right )^\beta\right ]^{-1}\\
\hphantom{\tilde{v}_{k,l}}{} \le 2^{\beta/2}\left ( {k\over k+2l}\right )^{\beta/2} \prod_{j=1}^l \left [1-\left ({k\over k+2j}\right )^\beta\right ]^{-1}.
\end{gather*}

For $k=2u, u\in\mathbb N$ we then get
\begin{gather}\label{eq:sigma-kl}
\tilde{v}_{2u,l}\le \zeta_{u,l}:=2^{\beta/2} \left ( {u\over u+l}\right )^{\beta/2} \prod_{j=1}^l \left [1-\left ({u\over u+j}\right )^\beta\right ]^{-1}
\end{gather}
and for $k=2u+1$, $u\in\mathbb N$
\begin{gather}\label{eq:odd}
\tilde{v}_{2u+1,l}=\left(\frac{u+1}{u+l+1}\right)^{\beta/2}\prod_{j=1}^l \left [1-\left(\frac{2u+1}{2u+2j+1}\right)^{\beta}\right]^{-1}<\zeta_{u+1,l}.
\end{gather}

Next, for any $N\in\mathbb N$
\begin{gather*}
\sum_{l=1}^\infty \zeta_{u,l}^2 \le
2^\beta \sum_{l=1}^\infty \left ( {u\over u+l}\right )^\beta\prod_{j=1}^l \left [1-\left ({u\over u+j}\right )^\beta\right ]^{-2} \\
\hphantom{\sum_{l=1}^\infty \zeta_{u,l}^2}{} =
2^\beta \sum_{l=1}^N \left ( {u\over u+l}\right )^\beta\prod_{j=1}^l \left [1-\left ({u\over u+j}\right )^\beta\right ]^{-2} \\
\hphantom{\sum_{l=1}^\infty \zeta_{u,l}^2\le }{} + 2^\beta \prod_{j=1}^N \left [1-\left ({u\over u+j}\right )^\beta\right ]^{-2}
\sum_{l=N+1}^\infty \left ( {u\over u+l}\right )^\beta\prod_{j=N+1}^l \left [1-\left ({u\over u+j}\right )^\beta\right ]^{-2}\\
\hphantom{\sum_{l=1}^\infty \zeta_{u,l}^2}{}\le 2^\beta (N+S_N) \prod_{j=1}^N \left [1-\left ({u\over u+j}\right )^\beta\right ]^{-2},
\end{gather*}
where
\begin{gather*}
S_N:=\sum_{l=N+1}^\infty \left ( {u\over u+l}\right )^\beta\prod_{j=N+1}^l \left [1-\left ({u\over u+j}\right )^\beta\right ]^{-2}.
\end{gather*}

Let now $N=u n_u$, where $n_u\in\mathbb N$ will be chosen later and write
\begin{gather*}
S_N = \sum_{n=n_u+1}^\infty \sum_{l=(n-1)u+1}^{nu}\left ( {u\over u+l}\right )^\beta\prod_{j=N+1}^l \left [1-\left ({u\over u+j}\right )^\beta\right ]^{-2}\\
\hphantom{S_N}{} \le u\sum_{n=n_u+1}^\infty \frac{1}{n^\beta}\prod_{j=N+1}^{nu} \left [1-\left ({u\over u+j}\right )^\beta\right ]^{-2}\\
\hphantom{S_N}{}= u\sum_{n=n_u+1}^\infty \frac{1}{n^\beta}\prod_{m=n_u+1}^n\; \prod_{j=(m-1)u+1}^{mu} \left [1-\left ({u\over u+j}\right )^\beta\right ]^{-2}\\
\hphantom{S_N}{} \le u\sum_{n=n_u+1}^\infty \frac{1}{n^\beta}\prod_{m=n_u+1}^n \left[1-\frac{1}{m^\beta}\right]^{-2u}
<u\sum_{n=n_u+1}^\infty \frac{1}{n^\beta}\prod_{m=n_u+1}^\infty \left[1-\frac{1}{m^\beta}\right]^{-2u}.
\end{gather*}
We use the estimate
\begin{gather*}
\sum_{n=n_u+1}^\infty\frac{1}{n^\beta}<\int_{n_u}^\infty\frac{{\rm d}x}{x^\beta}=\frac{1}{(\beta-1)n_u^{\beta-1}}\le \frac{1}{\beta-1}
\end{gather*}
to get
\begin{gather*}
S_N<\frac{u}{\beta-1}\prod_{m=n_u+1}^\infty \left[1-\frac{1}{m^\beta}\right]^{-2u},
\end{gather*}
hence
\begin{gather*}
\sum_{l=1}^\infty \zeta_{u,l}^2 \le 2^\beta u\prod_{j=1}^{un_u}\left[1-\left(\frac{u}{u+j}\right)^\beta\right]^{-2}\left(n_u+\frac{1}{\beta-1}\prod_{m=n_u+1}^\infty\left[1-\frac{1}{m^\beta}\right]^{-2u}\right).
\end{gather*}
For $n_u=u$ we get
\begin{gather*}
\sum_{l=1}^\infty \zeta_{u,l}^2 \le \frac{2^\beta \beta}{\beta-1} u^2\prod_{j=1}^{u^2}\left[1-\left(\frac{u}{u+j}\right)^\beta\right]^{-2}\prod_{m=u+1}^\infty\left[1-\frac{1}{m^\beta}\right]^{-2u},
\end{gather*}
hence
\begin{gather*}
\limsup_{u\to\infty}\left(\sum_{l=1}^\infty \zeta_{u,l}^2\right)^{1/u}\le \limsup_{u\to\infty}
\prod_{j=1}^{u^2}\left[1-\left(\frac{u}{u+j}\right)^\beta\right]^{-2/u}
\end{gather*}
because
\begin{gather*}
\lim_{u\to\infty}\prod_{m=u+1}^\infty\left[1-\frac{1}{m^\beta}\right]^{-2}=1.
\end{gather*}

Furthermore, for $u\to\infty$ we have
\begin{gather*}
-\frac{2}{u}\sum_{j=1}^{u^2} \log\big[1-(1+j/u)^{-\beta}\big]\to -2 \int_1^\infty \log\big[1-x^{-\beta}\big]{\rm d}x,
\end{gather*}
hence
\begin{gather}\label{eq:hpc3}
\limsup_{u\to\infty}\left(\sum_{l=1}^\infty \zeta_{u,l}^2\right)^{1/u} \le {\rm e}^{2v(\beta)}.
\end{gather}
From \eqref{eq:V_n1}, \eqref{eq:sigma-kl} and \eqref{eq:odd} we get
\begin{gather*}
V_{2u}^{1/(2u)} \le \tilde{V}_{2u}^{1/(2u)}\le \left(1+\sum_{l=1}^\infty \zeta_{u,l}^2\right)^{1/(4u)},\\
V_{2u+1}^{1/(2u+1)} \le \tilde{V}_{2u+1}^{1/(2u+1)}\le \left(1+\sum_{l=1}^\infty \zeta_{u+1,l}^2\right)^{1/(4u+2)}.
\end{gather*}
We next use that
\begin{gather*}
\zeta_{u,1}=\left(\frac{2u}{u+1}\right)^{\beta/2}\left[1-\left(\frac{u}{u+1}\right)^\beta\right]^{-1}\to \infty \qquad \text{for}\quad u\to\infty,
\end{gather*}
so \eqref{eq:hpc3} implies that \eqref{eq:Vn-logcc1} holds. The last assertion is easy to prove.
\end{proof}

\begin{thm}\label{thm:thm2b} Assume $(b_n)$ satisfies
\begin{gather*}
\frac{b_{n-1}}{b_n}\le {\rm e}^{-f(n)},\qquad n\ge n_0\ge 1,
\end{gather*}
where $f(n)>0$ for $n\ge n_0$ and
\begin{gather*}
\liminf_{n} nf(n)=\alpha>1.
\end{gather*}
The symmetric moment problem given by $(b_n)$ is indeterminate of order $\le 1/\alpha$ and the quanti\-ty~$V_n$ from \eqref{eq:Vn} satisfies
\begin{gather}\label{eq:rootVn1}
\limsup_n V_n^{1/n}\le {\rm e}^{v(\alpha)/2},
\end{gather}
where $v(\alpha)$ is defined in \eqref{eq:taubeta} and $v(\infty):=0$.
\end{thm}

\begin{proof}Let $1<\beta<\alpha$ be arbitrary. There exists $N\ge n_0$ such that
\begin{gather*}
nf(n)\ge \beta,\qquad n\ge N,
\end{gather*}
and it follows by Proposition~\ref{thm:qinc-indet} that the moment problem is indeterminate of order $\le 1/\beta$, hence of order $\le 1/\alpha$ since $\beta$ can be arbitrarily close to~$\alpha$. Defining
\begin{gather*}
\tilde{b}_n=\begin{cases} b_n,& 0\le n\le N,\\
\displaystyle b_N\exp\left[\beta\left(\frac{1}{N+1}+\frac{1}{N+2}+\dots +\frac{1}{n}\right)\right],& n>N,
\end{cases}
\end{gather*}
we have
\begin{gather*}
\frac{b_{n-1}}{b_n}\le \frac{\tilde{b}_{n-1}}{\tilde{b}_n},\qquad n\ge 1,
\end{gather*}
hence $V_n\le \tilde{V}_n$ by Proposition~\ref{thm:compV}.

By \eqref{eq:shiftV3} we get
\begin{gather}\label{eq:hpc5}
\limsup_n V_n^{1/n}\le \limsup_n \tilde{V}_n^{1/n}=\limsup_n\big(\tilde{V}_n^{(N)}\big)^{1/n}.
\end{gather}
Define now
\begin{gather*}
\widehat{b}_0 = b_N\exp\left[-\beta\left(1+\frac12+\dots +\frac{1}{N}\right)\right],\\
\widehat{b}_n = \widehat{b}_0\exp\left[\beta\left(1+\frac12+\dots +\frac{1}{n}\right)\right],\qquad n\ge 1.
\end{gather*}
Then $(\widehat{b}_n)$ satisfies $\widehat{b}_{n-1}/\widehat{b}_{n}=\exp(-\beta/n),\;n\ge 1$, so by \eqref{eq:shiftV3} and Theorem~\ref{thm:genlogcc} we get
\begin{gather}\label{eq:hpc6}
\limsup_n\big(\widehat{V}_n^{(N)}\big)^{1/n}=\limsup_n \widehat{V}_n^{1/n}\le {\rm e}^{v(\beta)/2}.
\end{gather}
Using $\widehat{b}_n^{(N)}=\tilde{b}_n^{(N)}$, we see by \eqref{eq:hpc5} and \eqref{eq:hpc6} that
\begin{gather*}
\limsup_n V_n^{1/n}\le {\rm e}^{v(\beta)/2}.
\end{gather*}
Since $1<\beta<\alpha$ is arbitrary, we finally get \eqref{eq:rootVn1} by continuity.
\end{proof}

\section[Existence of the matrix product $\mathcal A\mathcal H$]{Existence of the matrix product $\boldsymbol{\mathcal A\mathcal H}$}\label{section5}

In all of this section we consider a symmetric indeterminate moment problem given by a positive sequence $(b_n)$ necessarily satisfying $\sum 1/b_n<\infty$, cf.\ the beginning of Section~\ref{section4}. We shall give sufficient conditions ensuring property~(cs*).

Since the odd moments vanish, \eqref{eq:suffAH1} takes the form
\begin{gather}\label{eq:series}
\sum_{n\ge j/2}^\infty c_{2n-j}s_{2n}<\infty\qquad\text{for all}\quad j\ge 0.
\end{gather}

Using the expressions \eqref{eq:Un} and \eqref{eq:Vn} for $U_n$ and $V_n$ we find for $n\ge j+1$
 \begin{gather}\label{eq:UnVn}
c_{2n-j}s_{2n}= V_{2n-j}U_n\frac{b_0^2b_1^2\cdots b_{n-1}^2}{b_0b_1\cdots b_{2n-j-1}} =V_{2n-j}U_n\frac{b_0b_1\cdots b_{n-1}}{b_nb_{n+1}\cdots b_{2n-j-1}}.
\end{gather}

\begin{thm}\label{thm:AH-logcv2}Assume $(b_n)$ is eventually $q$-increasing, cf.\ Definition~{\rm \ref{thm:defqinc}}. For $j\ge 0$ we have
\begin{gather*}
\lim_{n\to \infty} (c_{2n-j}s_{2n})^{1/n}=0.
\end{gather*}
In particular the series \eqref{eq:series} is convergent, so the moment problem has property {\rm (cs*)} of Definition~{\rm \ref{thm:defAH}}.
\end{thm}

\begin{proof}By Theorem~\ref{thm:evtlogconvex} and Theorem~\ref{thm:Vevtlogconv} $(U_n)$, $(V_n)$ are bounded. By increasing~$n_0$ from \eqref{eq:qinc} if necessary, we may assume that $b_n\ge 1$ for $n\ge n_0$ and that $n_0>j$. For suitable $C>0$ and $n>n_0$ we find
\begin{gather*}
c_{2n-j}s_{2n} \le C{b_0b_1\cdots b_{n-1}\over b_nb_{n+1}\cdots b_{2n-j-1}} =Cb_0b_1\cdots b_{n_0-1}{b_{n_0}b_{n_0+1}\cdots b_{n-1}\over b_nb_{n+1}\cdots b_{2n-j-1}}\\
\hphantom{c_{2n-j}s_{2n}}{} \le C b_0b_1\cdots b_{n_0-1}{b_{n_0}b_{n_0+1}\cdots b_{n-1}\over b_nb_{n+1}\cdots b_{2n-n_0-1} }\le C b_0b_1\cdots b_{n_0-1}q^{(n-n_0)^2},
\end{gather*}
and therefore
\begin{gather*}
\lim_{n\to\infty}(c_{2n-j}s_{2n})^{1/n}=0.\tag*{\qed}
\end{gather*}\renewcommand{\qed}{}
\end{proof}

Now we turn to a case which includes eventually log-concave sequences $(b_n)$ of sufficient growth.

\begin{thm}\label{thm:evtlogconcave} Assume $(b_n)$ satisfies
\begin{gather*}
{b_{n-1}\over b_{n}} \le {\rm e}^{-f(n)}, \qquad n\ge n_0\ge 1,
\end{gather*}
 where $f(n)>0$ for $n\ge n_0$ and
 \begin{gather*}
\alpha=\liminf_{n} nf(n)>1.
\end{gather*}
For $j\ge 0$ we then have
\begin{gather*}
\limsup_{n\to\infty}(c_{2n-j}s_{2n})^{1/n} \le
 k(\alpha):= \begin{cases} 4^{-\alpha}\exp(2u(\alpha)+v(\alpha)), & 1<\alpha <\infty,\\
0, & \alpha=\infty,
\end{cases}
\end{gather*}
where $u(\alpha)$, $v(\alpha)$ are defined in \eqref{eq:mbeta}, \eqref{eq:taubeta}. In particular the series {\rm (\ref{eq:series})} is convergent when $k(\alpha)<1$, and in this case the moment problem has property {\rm (cs*)} of Definition~{\rm \ref{thm:defAH}}.
\end{thm}

\begin{proof}Let $j\ge 0$ be given and assume $1<\beta<\alpha$. By assumptions the sequence $(b_n)$ is strictly increasing for $n\ge n_0$ and tends to infinity. Thus, there exists $n_1\ge \max(n_0,j+1)$ such that
\begin{gather}\label{beta1}
b_n\ge 1,\qquad f(n)\ge {\beta \over n},\qquad n\ge n_1
\end{gather}
and for $n>n_1$ we then have
\begin{gather*}
 { b_0b_1\cdots b_{n-1}\over b_nb_{n+1}\cdots b_{2n-j-1}}\le b_0b_1\cdots b_{n_1-1}{b_{n_1}b_{n_1+1}\cdots b_{n-1}\over b_nb_{n+1}\cdots b_{2n-n_1-1}}.
\end{gather*}
By assumptions and by (\ref{beta1}) we get for $k\ge 1$
\begin{gather*}
{b_{n_1+k-1}\over b_{n+k-1}} \le {\rm e}^{-\sum\limits_{i=n_1+k}^{n+k-1} f(i)}\le {\rm e}^{-\beta \sum\limits_{i=n_1+k}^{n+k-1}{1\over i}} \le {\rm e}^{-\beta [\log (n+k)-\log (n_1+k)]},
\end{gather*}
and we have used \eqref{eq:hpd1} to get the last inequality.

Therefore we have
\begin{gather*}
 {b_{n_1}b_{n_1+1}\cdots b_{n-1}\over b_nb_{n+1}\cdots b_{2n-n_1-1}}\le {\rm e}^{-\beta \sum\limits_{k=1}^{n-n_1} [\log(n+k)-\log(n_1+k)]}.
\end{gather*}
As the function $x\longmapsto\log(n+x)-\log(n_1+x)$ is decreasing for $x>0,$ we get
\begin{gather*}
\sum_{k=1}^{n-n_1} [\log(n+k)-\log(n_1+k)]\ge \int_1^{n-n_1+1}[\log (n+x)-\log(n_1+x)]{\rm d}x\\
\hphantom{\sum_{k=1}^{n-n_1} [\log(n+k)-\log(n_1+k)]}{}=\int_{n+1}^{2n-n_1+1}\log x{\rm d}x -\int_{n_1+1}^{n+1}\log x{\rm d}x = 2n\log 2+{\rm o}(n).
\end{gather*}
Summarizing we obtain
\begin{gather*}
\limsup_{n\to\infty}\left ({ b_0b_1\cdots b_{n-1}\over b_nb_{n+1}\cdots b_{2n-j-1}}\right )^{1/n}\le 4^{-\beta}.
\end{gather*}
Since $1<\beta<\alpha$ is arbitrary we get
\begin{gather}\label{eq:factor}
\limsup_{n\to\infty}\left ({ b_0b_1\cdots b_{n-1}\over b_nb_{n+1}\cdots b_{2n-j-1}}\right )^{1/n}\le 4^{-\alpha}.
\end{gather}
Finally, by juxtaposing \eqref{eq:UnVn}, \eqref{eq:factor}, \eqref{eq:rootUn1} and \eqref{eq:rootVn1}, we get
the conclusion.
\end{proof}

\begin{rem}\label{thm:rem4} The function \eqref{eq:ka}, i.e.,
\begin{gather*}
 k(\alpha)=4^{-\alpha}\exp(2u(\alpha)+v(\alpha))
\end{gather*}
is decreasing for $\alpha>1$. Using the software Maple we get
\begin{gather*}
 k(1.68745)\approx 1.00001,\qquad k(1.68746)\approx 0,99997.
\end{gather*}
Thus the series \eqref{eq:series} is convergent for $\alpha\ge 1.68746$.

Examples with $\alpha=\infty$ are given in Remark~\ref{rem1}. For $b_n=(n+1)^c$ we have $\alpha=c$, so property (cs*) is satisfied for $c\ge 1.68746$.

In Section~\ref{section7} we will analyze the case $b_n=(n+1)^c$ by different methods and show that(cs*) actually holds for $c>3/2$, but (cs) does not hold for $1<c<3/2$.
\end{rem}

\begin{rem}Chen--Ismail \cite{Ch:I} considered the symmetrized version of the birth-and-death process with quartic rates studied in~\cite{B:V}. In~\cite{Ch:I} we have
 \begin{gather*}
b_n=2(n+1)\sqrt{(2n+1)(2n+3)},\qquad n\ge 0,
\end{gather*}
so Theorem~\ref{thm:evtlogconcave} holds with
\begin{gather*}
f(n)=\log(1+1/n)+\frac12\log\left(\frac{2n+3}{2n-1}\right),\qquad n\ge 1,
\end{gather*}
and therefore $\alpha:=\lim\limits_{n\to\infty} nf(n)=2$. This shows that (cs*) holds in this case.

More generally, consider the symmetrized version of the birth-and-death process with polynomial rates of degree $p\ge 3$, where $(b_n)$ is given in \cite[Proposition~3.2]{B:S3}. Take the special case $d_j=e_j-p/2$, $j=1,\ldots,p$ so that $(b_n)$ is eventually log-concave by \cite[Proposition~6.2]{B:S3} and hence eventually strictly increasing by Lemma~\ref{thm:lemBS2}. In this case it is easy to prove that Theorem~\ref{thm:evtlogconcave} yields $\alpha=p/2$, and therefore (cs*) holds when $p\ge 4$.

In the next section we shall look at the cubic case $p=3$.
\end{rem}

\section{The symmetrized cubic case}\label{section6}

In a series of papers Valent and coauthors have studied birth-and-death processes with polynomial rates leading to indeterminate Stieltjes moment problems, if the polynomial degree is at least three, see \cite{Va98} and references therein. In~\cite{Va98} Valent formulated conjectures about order and type of these Stieltjes problems. The conjecture about order was settled in Romanov~\cite{Ro}. Another proof was given in~\cite{B:S3} together with expressions for the type using multi-zeta values. These formulas have lead to a proof of the Valent conjecture about type in Bochkov \cite{Bo}.

Valent's conjectures were based on concrete calculation of order and type in a quartic case~\cite{B:V} and a cubic case~\cite{Va98}. In the latter Valent considered the Stieltjes moment problem associated with the birth-and-death process with the cubic rates
\begin{gather}\label{eq:cubicrates}
\lambda_n=(3n+1)(3n+2)^2,\qquad \mu_n=(3n)^2(3n+1),
\end{gather}
leading to the Stieltjes moment problem with recurrence coefficients from \eqref{eq:rec} given by
\begin{gather*}
a_n=\lambda_n+\mu_n,\qquad b_n=\sqrt{\lambda_n\mu_{n+1}},\qquad n\ge 0.
\end{gather*}

The purpose of this section is to estimate the moments in this cubic case. Using these estimates it was possible to prove that property (cs) does not hold for the corresponding symmetric moment problem. This was a surprise to us because (cs) turned out to be true in all examples, we had been able to calculate so far.

In order to distinguish the quantities associated with this special case from the general quantities, we shall equip them with a star $*$.

The recurrence coefficients from \eqref{eq:recsym} are
\begin{gather}\label{eq:cubicsym}
b_{2n}^{*}=\sqrt{\lambda_n}=\sqrt{3n+1}(3n+2),\qquad b_{2n-1}^{*}=\sqrt{\mu_n}=3n\sqrt{3n+1},
\end{gather}
cf.~\cite[p.~343]{B:S3}.

\begin{thm}\label{thm:AHnotex} For the symmetric moment problem given by $(b_n^*)$ and corresponding to the Stieltjes problem with rates~\eqref{eq:cubicrates} we have
\begin{gather*}
\sum_{n=0}^\infty c_{2n}^*s_{2n}^*=\infty,
\end{gather*}
so property {\rm (cs)} does not hold. The series
\begin{gather*}
\sum_{l=0}^\infty a_{2k,2l}^* s_{2l+2m}^*
\end{gather*}
is not absolutely convergent for $k\ge 0$, $m\ge 1$, so property {\rm (ac)} does not hold.
\end{thm}

The proof will be given after Proposition~\ref{thm:ck*}.

\subsection[Estimate of the moments $s_{2n}^*$]{Estimate of the moments $\boldsymbol{s_{2n}^*}$}
In accordance with \cite{Va98} define
\begin{gather}\label{eq:theta}
\theta:=\int_0^1\frac{{\rm d}u}{(1-u^3)^{2/3}}=\frac13 B(1/3,1/3)=\frac{\Gamma^3(1/3)}{2\pi\sqrt{3}},\qquad a:=\theta^3,
\end{gather}
where $B$ is the beta-function. Note that
\begin{gather*}
 \theta\approx 1.76663,\qquad a\approx 5.51370,\qquad \Gamma(1/3)\approx 2.67893.
\end{gather*}
The Stieltjes problem given by \eqref{eq:cubicrates} is indeterminate and Valent found the $B$ and $D$ functions from the Nevanlinna matrix. To give his results let
\begin{gather*}
\varphi_0(z)=\sum_{n=0}^\infty \frac{(-1)^n}{(3n)!}z^n,\qquad \varphi_2(z)=\sum_{n=0}^\infty \frac{(-1)^n}{(3n+2)!}z^{n}.
\end{gather*}

They are entire functions of order 1/3 and type 1.

Valent's result is that
\begin{gather}\label{eq:BDcubic}
D(z)=c z\varphi_2(az),\qquad B(z)-\frac{1}{t_0}D(z)=-\varphi_0(az),
\end{gather}
where
\begin{gather*}
c=\frac{\sqrt{3}}{(2\pi)^3}\Gamma(1/3)^6\approx 2.58105,
\end{gather*}
and $t_0$ is a certain negative constant defined as
\begin{gather*}
t_0=\lim_{x\to-\infty}D(x)/B(x),
\end{gather*}
cf.~\cite[Lemma~2.2.1]{B:V}, where it was called~$\alpha$. Valent deduced from~\eqref{eq:BDcubic} that the order and type of the Stieltjes problem are~$1/3$ and~$\theta$.

Associated with the constant $t+i\gamma$, where $t\in\mathbb{R}$, $\gamma>0$, there is a solution to the moment problem with the density on $\mathbb{R}$, cf.~\cite{B:C} and \cite[equation~(2.17)]{B:V},
\begin{gather*}
\nu_{t+{\rm i}\gamma}(x)=\frac{\gamma}{\pi}\frac{1}{(tB(x)-D(x))^2+\gamma^2B^2(x)}.
\end{gather*}
We now use
\begin{gather*}
t=\frac{t_0}{1+t_0^2},\qquad \gamma=\frac{t_0^2}{1+t_0^2}
\end{gather*}
leading to
\begin{gather*}
\nu_{t+{\rm i}\gamma}(x)=\frac{1}{\pi}\frac{1}{D^2(x)+(B(x)-D(x)/t_0)^2},
\end{gather*}
hence to a density for the Stieltjes problem with moments
\begin{gather*}
s_n=\frac{1}{\pi}\int_{-\infty}^\infty \frac{x^n}{\varphi_0^2(ax)+c^2 x^2\varphi_2^2(ax)}{\rm d}x,\qquad n\ge 0.
\end{gather*}

Let $B_s$, $D_s$ denote the functions from the Nevanlinna matrix in the symmetrized case. Then we know (cf., e.g.,~\cite[p.~48]{Be95})
\begin{gather*}
B_s(z)=B\big(z^2\big)-\frac{1}{t_0}D\big(z^2\big)=-\varphi_0\big(az^2\big),\qquad D_s(z)=(1/z)D\big(z^2\big)=c z\varphi_2\big(az^2\big).
\end{gather*}
This gives the symmetric density
\begin{gather*}
\big(\pi\big(B_s^2(x)+D_s^2(x)\big)\big)^{-1}=\frac{1}{\pi}\big(\varphi_0^2\big(ax^2\big)+c^2x^2\varphi_2^2\big(ax^2\big)\big)^{-1}
\end{gather*}
for the symmetric problem with even moments $s_{2n}^*$ and odd moments~0, i.e.,
\begin{gather}
s_{2n}^* = \frac{1}{\pi}\int_{-\infty}^\infty\frac{x^{2n}}{\varphi_0^2\big(ax^2\big)+c^2x^2\varphi_2^2\big(ax^2\big)}{\rm d}x
 = \frac{a^{-n-1/2}}{\pi}\int_0^\infty\frac{t^{n-1/2}}{\varphi_0^2(t)+(c^2/a)t\varphi_2^2(t)}{\rm d}t.\label{eq:s2n}
\end{gather}

Using \eqref{eq:s2n} we shall prove the following estimates of $s_{2n}^*$.

\begin{prop}\label{thm:s2n} There exist constants $K_i>0$, $i=1,2$, such that
\begin{gather}\label{eq:final}
K_1a^{-n}\Gamma(3n+3/2)<s_{2n}^*<K_2 a^{-n}n\Gamma(3n+3/2),\qquad n\ge 1,
\end{gather}
where $a$ is given by \eqref{eq:theta}. In particular
\begin{gather}\label{eq:final1}
\lim_{n\to\infty}\frac{\root{n}\of {s_{2n}^*}}{n^3}=\frac{27}{a{\rm e}^3}=\left(\frac{3}{{\rm e}\theta}\right)^3
\end{gather}
and
\begin{gather}\label{eq:Uncub}
\lim_{n\to\infty} \root{n}\of {U_n^*}=\left(\frac{2}{\theta}\right)^3.
\end{gather}
\end{prop}

\begin{proof}In \cite{G:L:V} some interesting formulas are given about the trigonometric functions of order~3. These functions are defined as
\begin{gather*}
\sigma_l(u)=\sum_{n=0}^\infty (-1)^n\frac{u^{3n+l}}{(3n+l)!},\qquad l=0,1,2.
\end{gather*}
With ${\rm j}={\rm e}^{{\rm i}\pi/3}$ we have
\begin{gather*}
\sigma_0(u) = \frac13\big({\rm e}^{-u}+{\rm e}^{{\rm j}u}+{\rm e}^{\overline{{\rm j}}u}\big)=\frac13\left({\rm e}^{-u}+2\cos\left(\frac{\sqrt{3}}{2}u\right){\rm e}^{u/2}\right)=\varphi_0\big(u^3\big),\\
\sigma_2(u) = \frac13\big({\rm e}^{-u} -{\rm j} {\rm e}^{ju}- \overline{{\rm j}}{\rm e}^{\overline{{\rm j}}u}\big)=\frac13\left({\rm e}^{-u}-2\cos\left(\frac{\sqrt{3}}{2}u+\pi/3\right){\rm e}^{u/2}\right) = u^2\varphi_2\big(u^3\big),
\end{gather*}
and there is a similar expression for $\sigma_1(u)$, which we do not need.

Substituting $t=u^3$ in \eqref{eq:s2n} we get
\begin{gather*}
s_{2n}^*=\frac{3a^{-n-1/2}}{\pi}\int_0^\infty\frac{u^{3n+1/2}}{\sigma_0^2(u)+\big(c^2/(au)\big)\sigma_2^2(u)}{\rm d}u
\end{gather*}
or
\begin{gather*}
s_{2n}^*=\frac{27a^{-n-1/2}}{\pi}\int_0^\infty\frac{u^{3n+1/2}{\rm e}^{-u}}{N(u)}{\rm d}u
\end{gather*}
with
\begin{gather*}
N(u)=\left({\rm e}^{-3u/2}+2\cos\left(\frac{\sqrt{3}}{2}u\right)\right)^2+\frac{c^2}{au}
\left({\rm e}^{-3u/2}-2\cos\left(\frac{\sqrt{3}}{2}u+\pi/3\right)\right)^2.
\end{gather*}
One can prove that $N(u)\le 9$ for $u\ge 0$. (The constant comes from the software Maple. That~$N(u)$ is bounded for $u\ge 0$ is however elementary, since the first term is bounded by~9, and using that
\begin{gather*}
{\rm e}^{-3u/2}-2\cos\left(\frac{\sqrt{3}}{2}u+\pi/3\right)
\end{gather*}
vanishes for $u=0$, also the second term is bounded.)

Furthermore,
\begin{gather*}
N(u)\ge k_1:=\big({\rm e}^{-3/2}+2\cos\big(\sqrt{3}/2\big)\big)^2\approx 2.30690,\qquad 0\le u\le 1,
\end{gather*}
and using that $c^2/a\approx 1.20823>1$, we get for $u\ge 1$:
\begin{gather*}
uN(u) > \left({\rm e}^{-3u/2}+2\cos\left(\frac{\sqrt{3}}{2}u\right)\right)^2+ \left({\rm e}^{-3u/2}-2\cos\left(\frac{\sqrt{3}}{2}u+\pi/3\right)\right)^2\\
\hphantom{uN(u)}{} = 2{\rm e}^{-3u}+4\cos^2\left(\frac{\sqrt{3}}{2}u\right)+4\cos^2\left(\frac{\sqrt{3}}{2}u+\pi/3\right)\\
\hphantom{uN(u)>}{} + 4{\rm e}^{-3u/2}\left(\cos\left(\frac{\sqrt{3}}{2}u\right)-\cos\left(\frac{\sqrt{3}}{2}u+\pi/3\right)\right)\\
\hphantom{uN(u)}{} \ge f(u):=2{\rm e}^{-3u}-8{\rm e}^{-3u/2}+2\ge f(1),
\end{gather*}
where we have used that $\cos^2(x)+\cos^2(x+\pi/3)\ge 1/2$ and that $f(u)$ is increasing (for $u\ge 0$). This shows that
\begin{gather*}
uN(u)\ge k_2:=f(1)\approx 0.31453, \qquad u\ge 1.
\end{gather*}

We then get
\begin{gather*}
 \frac{3a^{-n-1/2}}{\pi}\Gamma(3n+3/2)<s_{2n}^*<\frac{27a^{-n-1/2}}{\pi}\left(\int_0^1
\frac{u^{3n+1/2}{\rm e}^{-u}}{k_1}{\rm d}u +\int_1^\infty\frac{u^{3n+3/2}{\rm e}^{-u}}{k_2}{\rm d}u\right)\\
\hphantom{\frac{3a^{-n-1/2}}{\pi}\Gamma(3n+3/2)<s_{2n}^*}{} <\frac{27a^{-n-1/2}}{\pi}\left(\frac{1}{k_1}+\frac{1}{k_2}\Gamma(3n+5/2)\right),
\end{gather*}
which shows \eqref{eq:final}.

Note that by Stirling's formula
\begin{gather}\label{eq:Stirlingdelta}
\Gamma(3n+\delta)\sim 3^{\delta-1/2} \sqrt{2\pi}\big(27/{\rm e}^3\big)^n n^{3n+\delta-1/2},
\end{gather}
so \eqref{eq:final} leads to \eqref{eq:final1}.

From \eqref{eq:cubicsym} we get
\begin{gather}\label{eq:prodsym}
b_0^{*}b_1^{*}\cdots b_{2n}^{*}=\Gamma(3n+2),\qquad b_0^{*}b_1^{*}\cdots b_{2n-1}^{*}=\sqrt{3n+1}\Gamma(3n+1),
\end{gather}
which together with \eqref{eq:final1} gives \eqref{eq:Uncub}.
\end{proof}

\subsection[The matrix $\mathcal A$ in the symmetrized cubic case]{The matrix $\boldsymbol{\mathcal A}$ in the symmetrized cubic case}

Let $(S_n)$ denote the orthonormal polynomials in the symmetrized cubic case. They are given by~\eqref{eq:recsym(s)}, with ${}^sb_n=b_n^*$ from~\eqref{eq:cubicsym}.

 It is known that, cf.~\eqref{eq:repksym},
\begin{gather*}
K_s(z,w) = \sum_{n=0}^\infty S_n(z)S_n(w)=\frac{D_s(z)B_s(w)-B_s(z)D_s(w)}{w-z} = \sum_{k,l=0}^\infty a_{k,l}^{(s)}z^kw^l,
\end{gather*}
where the third expression can be found in \cite[equation~(1.8)]{Be95}.

We shall find explicit formulas for $a_{k,l}^{(s)}$ based on the following lemma.

\begin{lem}\label{thm:kernel} Let $f$, $g$ be entire functions with power series
\begin{gather*}
f(z)=\sum_{n=0}^\infty \alpha_nz^n,\qquad g(z)=\sum_{n=0}^\infty \beta_nz^n.
\end{gather*}
Then the entire function of two variables
\begin{gather*}
L(z,w)=\frac{f(z)g(w)-f(w)g(z)}{z-w}
\end{gather*}
has the power series
\begin{gather*}
L(z,w)=\sum_{k,l=0}^\infty \gamma_{k,l}z^kw^l,
\end{gather*}
where
\begin{gather}\label{eq:ker}
\gamma_{k,l}=\sum_{j=0}^l (\alpha_{j+k+1}\beta_{l-j}-\beta_{j+k+1}\alpha_{l-j}).
\end{gather}
\end{lem}

\begin{rem}\label{rem5}By construction $L(z,w)=L(w,z)$ and therefore $\gamma_{k,l}=\gamma_{l,k}$, but this is not obvious from~\eqref{eq:ker}.
\end{rem}

\begin{proof}We have
\begin{gather*}
L(z,w)=\frac{f(z)-f(w)}{z-w} g(w)-f(w)\frac{g(z)-g(w)}{z-w}
\end{gather*}
and
\begin{gather*}
\frac{f(z)-f(w)}{z-w} = \sum_{n=1}^\infty\alpha_n\frac{z^n-w^n}{z-w}=\sum_{n=1}^\infty\alpha_n \sum_{j=0}^{n-1}z^{n-1-j}w^j\\
\hphantom{\frac{f(z)-f(w)}{z-w}}{} = \sum_{j=0}^\infty w^j\sum_{n=j+1}^\infty \alpha_n z^{n-1-j}=\sum_{j,k=0}^\infty\alpha_{j+k+1}z^kw^j.
\end{gather*}
Therefore,
\begin{gather*}
\frac{f(z)-f(w)}{z-w} g(w) = \sum_{m=0}^\infty \beta_mw^m\sum_{j,k=0}^\infty\alpha_{j+k+1}z^kw^j = \sum_{k,l=0}^\infty z^kw^l\left(\sum_{j=0}^l \alpha_{j+k+1}\beta_{l-j}\right),
\end{gather*}
and by symmetry we get \eqref{eq:ker}.
\end{proof}

Write
\begin{gather*}
f(z)=B_s(z)=\sum_{n=0}^\infty\alpha_n z^n,\qquad \alpha_{2n}=\frac{-(-a)^{n}}{(3n)!},\qquad \alpha_{2n+1}=0,\\
g(z)=D_s(z)= \sum_{n=0}^\infty\beta_nz^n,\qquad \beta_{2n}=0,\qquad \beta_{2n+1}=\frac{c(-a)^n}{(3n+2)!}.
\end{gather*}
Then $K_s(z,w)=L(z,w)$ from Lemma~\ref{thm:kernel}, and we get
\begin{gather*}
a_{k,l}^{(s)}=a_{k,l}^*=\sum_{j=0}^l (\alpha_{j+k+1}\beta_{l-j}-\beta_{j+k+1}\alpha_{l-j}).
\end{gather*}

In particular $a_{0,0}^*=-\alpha_0\beta_1=c/2$ and for $k\ge 1$
\begin{gather*}
a_{k,k}^*=\sum_{j=0}^k(\alpha_{j+k+1}\beta_{k-j}-\beta_{j+k+1}\alpha_{k-j}).
\end{gather*}
Let us first focus on $k$ even ($\ge 2$):
\begin{gather*}
a_{2k,2k}^* = \sum_{j=0}^{2k}(\alpha_{j+2k+1}\beta_{2k-j}-\beta_{j+2k+1}\alpha_{2k-j}) = \sum_{j=0}^{k-1} \alpha_{2j+2k+2}\beta_{2k-2j-1}-\sum_{j=0}^{k}\beta_{2j+2k+1}\alpha_{2k-2j}\\
\hphantom{a_{2k,2k}^*}{} = c\sum_{j=0}^{k-1} \frac{-(-a)^{j+k+1}}{(3j+3k+3)!}\frac{(-a)^{k-j-1}}{(3k-3j-1)!} -c\sum_{j=0}^{k} \frac{(-a)^{j+k}}{(3k+3j+2)!}\frac{-(-a)^{k-j}}{(3k-3j)!}\\
\hphantom{a_{2k,2k}^*}{}=c a^{2k}\left(\frac{1}{(6k+2)!}+\sum_{j=0}^{k-1}\frac{6j+3}{(3k+3j+3)!(3k-3j)!}\right),
\end{gather*}
hence
\begin{gather}\label{eq:even}
a_{2k,2k}^*=\frac{ca^{2k}}{(6k+3)!}\sum_{j=0}^k (6j+3)\binom{6k+3}{3k-3j}.
\end{gather}

For $k$ odd we similarly get
\begin{gather*}
a_{2k-1,2k-1}^*=\frac{6ca^{2k-1}k}{(6k+1)!}\sum_{j=0}^{k-1} (6j+3)\binom{6k+1}{3k-3j-1}.
\end{gather*}

Using the following estimates
\begin{gather*}
3\binom{6k+3}{3k} < \sum_{j=0}^k (6j+3)\binom{6k+3}{3k-3j} < \binom{6k+3}{3k}\sum_{j=0}^k (6j+3)=3(k+1)^2\binom{6k+3}{3k},
\end{gather*}
we get for suitable $C_i>0$, $i=1,2$,
\begin{gather}\label{eq:c2k}
C_1 \big(a {\rm e}^3/27\big)^k k^{-3k-2} <c_{2k}^*=\sqrt{a_{2k,2k}^*}< C_2 \big(a {\rm e}^3/27\big)^k k^{-3k}.
\end{gather}
From \eqref{eq:c2k} and similar inequalities for $c_{2k-1}^*=\sqrt{a_{2k-1,2k-1}^*}$ together with~\eqref{eq:prodsym}, we get

\begin{prop}\label{thm:ck*}
\begin{gather}\label{eq:ck*}
\lim_{n\to\infty} n^{3/2}\root{n}\of{c_{n}^*}=\left(\frac{2{\rm e}\theta}{3} \right)^{3/2}.
\end{gather}
and
\begin{gather}\label{eq:Vkcub}
\lim_{n\to\infty} \root{n} \of{V_n^*}=\theta^{3/2}.
\end{gather}
\end{prop}

Equation \eqref{eq:ck*} is a sharpening of equation~\eqref{eq:otcn}, by replacing $\limsup$ by a limit. Note that $\rho=2/3$ is the order and $\tau=\theta$ is the type of the symmetric moment problem. Equation \eqref{eq:Vkcub} is a sharpening of \eqref{eq:rootVn1} applied to $b_n^{*}$ given by~\eqref{eq:cubicsym}.

\begin{proof}[Proof of Theorem~\ref{thm:AHnotex}] Putting \eqref{eq:final} and \eqref{eq:c2k} together and using \eqref{eq:Stirlingdelta}, we get
\begin{gather*}
 L_1n^{-1} < s_{2n}^*c_{2n}^* <L_2 n^2
\end{gather*}
for suitable $L_i>0$, $i=1,2$.

This shows divergence of $\sum s_{2n}^*c_{2n}^*$, which is the first part of Theorem~\ref{thm:AHnotex}.

We shall next show that most of the series involved in defining $\mathcal A\mathcal H$ are not absolutely convergent.
We know that $a_{2k,2l+1}^*=0$ and equation~\eqref{eq:even} can easily be generalized to
\begin{gather*}
a_{2k,2l}^*=\frac{c(-a)^{k+l}}{(3k+3l+3)!}\sum_{j=0}^l \binom{3k+3l+3}{3l-3j}(3k-3l+6j+3).
\end{gather*}

By symmetry we can assume that $k\ge l$, so all terms in the sum are positive. We then get
\begin{gather*}
(3(k-l)+3)\binom{3k+3l+3}{3l} < \sum_{j=0}^l (3(k-l)+6j+3)\binom{3k+3l+3}{3l-3j}\\
\hphantom{(3(k-l)+3)\binom{3k+3l+3}{3l}}{} < 3(k+1)(l+1)\binom{3k+3l+3}{3l},
\end{gather*}
because
\begin{gather*}
\binom{3k+3l+3}{3l-3j} \le \binom{3k+3l+3}{3l}, \qquad 0\le j\le l\le k,
\end{gather*}
hence
\begin{gather*}
c a^{k+l}\frac{3(k-l)+3}{(3k+3)!(3l)!}< |a_{2k,2l}^*|<c a^{k+l}\frac{3(k+1)(l+1)}{(3k+3)!(3l)!}.
\end{gather*}
If we look at $\sum_k |a_{2k,2l}^*|s_{2k+2m}^*$, for each fixed $l$, $m$, it is enough to examine convergence of the series for $k>l$, i.e., of the terms
\begin{gather*}
\frac{a^k k}{(3k+2)!}s_{2k+2m}^*.
\end{gather*}
Using \eqref{eq:final}, we get that for $k$ sufficiently large
\begin{gather*}
M_1k^{3m-3/2} < \frac{a^k k}{(3k+2)!}s_{2k+2m}^* < M_2k^{3m-1/2}
\end{gather*}
for constants $M_1$, $M_2$ depending on $l$, $m$ but independent of~$k$. This shows that the series $\sum_k |a_{2k,2l}^*|s_{2k+2m}^*$ diverges for $m\ge 1$, but we do not get information about convergence/diver\-gen\-ce when $m=0$.

We conclude that most of the series involved in the calculation of $\mathcal A\mathcal H$ are not absolutely convergent.
\end{proof}

\section[The symmetric moment problem with $b_n=c^c(n+1)^c$, $c>0$]{The symmetric moment problem with $\boldsymbol{b_n=c^c(n+1)^c}$, $\boldsymbol{c>0}$}\label{section7}

The sequence $b_n=b_n(c)=c^c(n+1)^c$ is log-concave for $c>0$. By the Carleman condition the moment problem is determinate for $0<c\le 1$, and for $c>1$ it is indeterminate of order $\rho=1/c$, cf.~\cite[Theorem~4.11]{B:S2}. The factor $c^c$ does not influence the order of the moment problem, but it is convenient when $c=3/2$ for comparison with the symmetrized cubic case.

The corresponding orthonormal polynomials $P_n(c;x)$ do not seem to be known except in the cases $c=1/2$ and $c=1$, where they are scaled versions of the Hermite polynomials and Meixner--Pollaczek polynomials. In a precise manner
\begin{gather*}
P_n(1/2;x)=(2^nn!)^{-1/2}H_n\big(x/\sqrt{2}\big),\qquad P_n(1;x)=P_n^{(1/2)}(x/2;\pi/2),
\end{gather*}
where $H_n$ are the Hermite polynomials and $P_n^{(\lambda)}(x;\varphi)$ the Meixner--Pollaczek polynomials in the notation of \cite{K:S}.

\begin{defn}\label{thm:type-c} The type of the symmetric indeterminate moment problem with $b_n=c^c(n+1)^c$, $c>1$ is denoted by $\tau_c$.
\end{defn}

By \cite[Proposition~2.3]{B:S3} we know that $b_n=(n+1)^c$ leads to a moment problem with order~$1/c$ and type
\begin{gather*}
\frac{\tau_c}{(c^{-c})^{1/c}}=c\tau_c.
\end{gather*}

For real $c>1$ and $n\ge 1$ define the multi-zeta value\footnote{Observe that the inequalities between the indices $k_j$ are alternating between $\le$ and $<$: $k_{2j-1}\le k_{2j}<k_{2j+1}\le k_{2j+2}$.}
\begin{gather*}%\label{eq:mz}
\gamma_n(c)=\sum_{1\le k_1\le k_2<\cdots<k_{2n-1}\le k_{2n}}\left(k_1k_2\cdots k_{2n-1}k_{2n}\right)^{-c}.
\end{gather*}
Next define the function
\begin{gather*}
G_c(z)=\sum_{n=1}^\infty \gamma_n(c) z^n.
\end{gather*}

By \cite[Remark 1.6]{B:S3} it is known that the order of $G_c$ is $1/(2c)$ and by \cite[Theorem 1.11]{B:S3} we get $c\tau_c=T_c/2$, where $T_c$ is the type of $G_{c}$ given by
\begin{gather*}
T_c=\frac{2c}{e}\limsup_{n\to\infty}\big(n(\gamma_n(c))^{1/(2cn)}\big),
\end{gather*}
based on \cite{Le}. (Unfortunately there is a misprint in equation~(18) in \cite[p.~339]{B:S3}.) By Bochkov's proof in~\cite{Bo} of the Valent conjecture about type follows that
\begin{gather*}
T_c=B\left(\frac{1}{2c},1-\frac{1}{c}\right),\qquad \tau_c=\int_0^1\frac{{\rm d}t}{\big(1-t^{2c}\big)^{1/c}}.
\end{gather*}

We shall first compare quantities defined in terms of $b_n=b_n(3/2)=(3/2)^{3/2}(n+1)^{3/2}$ with quantities in Section~\ref{section6} for the symmetrized cubic case, where $b_n=b_n^*$ is defined in~\eqref{eq:cubicsym}.

\begin{lem}\label{thm:3/2-cub}Let $s_{2n},c_n$ denote the quantities corresponding to $b_n= b_n(3/2)$. Then
\begin{gather}\label{eq:3/2-s2n}
\lim_{n\to\infty}\frac{\root{n}\of{s_{2n}}}{n^3}=\left(\frac{3}{{\rm e}\theta}\right)^3,
\end{gather}
$($i.e., the same value as in \eqref{eq:final1}$)$ and
\begin{gather}\label{eq:3/2-cn}
\limsup_{n\to\infty} n^{3/2}\root{n}\of{c_n}=\limsup_{n\to\infty} (2n)^{3/2}\root{2n}\of{c_{2n}}=\left(\frac{2 {\rm e} \theta}{3}\right)^{3/2},
\end{gather}
$($i.e., the same value as in \eqref{eq:ck*}$)$. In particular
\begin{gather}\label{eq:3/2-cn-sn}
\limsup_{n\to\infty}\root{n} \of{c_{2n}s_{2n}}=1.
\end{gather}
\end{lem}

\begin{proof} First note that
\begin{gather}\label{eq:3bs}
b_n(3/2)\le b_n^* \le b_{n+1}(3/2),
\end{gather}
i.e.,
\begin{gather}\label{eq:3bsa}
(3/2)^{3/2}(2n+1)^{3/2} \le \sqrt{3n+1}(3n+2) \le (3/2)^{3/2}(2n+2)^{3/2},\\
(3/2)^{3/2}(2n)^{3/2} \le 3n\sqrt{3n+1} \le (3/2)^{3/2}(2n+1)^{3/2}\label{eq:3bsb}.
\end{gather}
The two inequalities to the right are obvious, and so is the inequality to the left in \eqref{eq:3bsb}. The inequality to the left in \eqref{eq:3bsa} is equivalent to
\begin{gather*}
(n+1/2)^3 \le (n+1/3)(n+2/3)^2.
\end{gather*}
Setting $x=n+1/2$ the inequality is reduced to
\begin{gather*}
x^3\le (x-1/6)(x+1/6)^2= x^3+ \frac{1}{36}\big(\big(2x^2-x\big)+\big(4x^2-1/6\big)\big),
\end{gather*}
which is valid for $x\ge 1/2$.

By \eqref{eq:3bs} we get from Proposition~\ref{thm:increasofbn}, \eqref{eq:shiftU1} and~\eqref{eq:Un+1}
\begin{gather*}
s_{2n}\le s_{2n}^* \le s_{2n}^{(1)} \le \left(\frac23\right)^3 s_{2n+2}\le 4(n+1)^3 s_{2n}.
\end{gather*}
It follows that
\begin{gather*}
\frac{\root{n}\of{s_{2n}}}{n^3} \le \frac{\root{n}\of{s_{2n}^*}}{n^3} \le \big(4(n+1)^3\big)^{1/n}\frac{\root{n}\of{s_{2n}}}{n^3}
\end{gather*}
and \eqref{eq:3/2-s2n} follows from \eqref{eq:final1}.

The function $\sum\limits_{n=0}^\infty c_nz^n$ has the same order and type as the moment problem for $(b_n(3/2))$ by Theorem~3.1 in~\cite{B:S2}, so by~\eqref{eq:otcn}
\begin{gather*}
\limsup_{n\to\infty} n^{3/2} \root{n}\of{c_n} = \left(\frac{2 {\rm e}\tau_{3/2}} {3}\right)^{3/2}.
\end{gather*}
 However, by Remark~1.10 in~\cite{B:S3} we have $\tau_{3/2}=\theta$, where $\theta$ is the type of the cubic Stieltjes case of Section~\ref{section6}, and the type is unchanged when we consider the symmetrized version, cf.\ Proposition~3.1 in~\cite{B:S3}.

Therefore
\begin{gather*}
\limsup_{n\to\infty} n^{3/2}\root{n}\of{c_n} = \left(\frac{2{\rm e}\theta}{3}\right)^{3/2}= \limsup_{n\to\infty} n^{3/2}\root{n}\of{c_n^*},
\end{gather*}
where the last equality sign comes from \eqref{eq:ck*}.

By Proposition~\ref{thm:aV_k} we have $c_k\le b_k(3/2)c_{k+1}$, and from this it is easy to see that
\begin{gather*}
\limsup_{n\to\infty}(2n)^{3/2}\root{2n}\of{c_{2n}} = \limsup_{n\to\infty}(2n+1)^{3/2}\root{2n+1}\of{c_{2n+1}},
\end{gather*}
and \eqref{eq:3/2-cn} follows.
\end{proof}

\begin{thm}\label{thm:c>3/2}Define $b_n(c)=c^c(n+1)^c$, $c>1$ and let $s_{2n}(c)$, $c_n(c)$ denote the correspon\-ding quantities. Then the series $\sum c_{2n}(c) s_{2n}(c)$ is convergent for $c>3/2$ and divergent for \smash{$1<c<3/2$}, i.e., property {\rm(cs)} holds for $c>3/2$, but not for $1<c<3/2$.
\end{thm}

\begin{proof} From \eqref{eq:UnVn} we get
\begin{gather*}
c_{2n}(c)s_{2n}(c)\binom{2n}{n}^c=V_{2n}(c)U_n(c),
\end{gather*}
hence
\begin{gather*}
\frac{c_{2n}(c)s_{2n}(c)}{c_{2n}(3/2)s_{2n}(3/2)}=\frac{V_{2n}(c)U_n(c)}{V_{2n}(3/2)U_n(3/2)}\binom{2n}{n}^{3/2-c}.
\end{gather*}
By Propositions~\ref{thm:comparison} and~\ref{thm:compV} $U_n(c)$ and $V_n(c)$ are decreasing in~$c$, hence
\begin{gather*}
\frac{c_{2n}(c)s_{2n}(c)}{c_{2n}(3/2)s_{2n}(3/2)}\ \begin{cases}
\le \binom{2n}{n}^{3/2-c} & \text{for} \ 3/2 < c,\\
\ge \binom{2n}{n}^{3/2-c} & \text{for} \ 1 < c<3/2.
\end{cases}
\end{gather*}

By Stirling's formula
\begin{gather*}
\lim_{n\to\infty} \binom{2n}{n}^{1/n}=4,
\end{gather*}
so by \eqref{eq:3/2-cn-sn} we get
\begin{gather*}
\limsup_{n\to\infty}\root{n} \of{c_{2n}(c)s_{2n}(c)} \ \begin{cases} \le 4^{3/2-c} < 1& \text{for} \ 3/2 < c, \\
\ge 4^{3/2-c} >1& \text{for} \ 1 < c < 3/2.
\end{cases}\tag*{\qed}
\end{gather*}\renewcommand{\qed}{}
\end{proof}

\begin{thm}\label{thm:AH-c>3/2} The symmetric moment problem with $b_n(c)=c^c(n+1)^c$ has property~{\rm (cs*)} of Definition~{\rm \ref{thm:defAH}} when $c>3/2$.
\end{thm}

\begin{proof} By Proposition~\ref{thm:aV_k} we have $c_k(c)\le b_k(c)c_{k+1}(c)$, hence for $j\ge 0$
\begin{gather*}
c_{2n-j}(c)\le c^{cj}\big((2n-j+1)(2n-j+2)\cdots(2n)\big)^c c_{2n}(c),
\end{gather*}
and it follows that
\begin{gather*}
\limsup_{n\to\infty} \root{n} \of{c_{2n-j}(c)s_{2n}(c)}\le 4^{3/2-c}<1.
\end{gather*}
This shows that \eqref{eq:series} holds.
\end{proof}

\begin{rem}\label{thm:rem75} We do not know if (ac) or (cs) holds in case $c=3/2$. We also do not know if (aci) holds when $1<c<3/2$. In order to answer these questions, it seems to be necessary to have information about the orthogonal polynomials~$P_n(3/2;x)$.
\end{rem}

\appendix
\section{Appendix}\label{appendixA}

\subsection[A lower bound of $(U_n)$ when $(b_n)$ is bounded]{A lower bound of $\boldsymbol{(U_n)}$ when $\boldsymbol{(b_n)}$ is bounded}

\begin{prop}\label{mean} Let $(b_n)$ be a bounded sequence. Then
\begin{gather*}
\liminf_{n\to\infty} U_n^{1/n}\ge 4.
\end{gather*}
If $\lim b_n=b\ge 0$, then $\liminf$ is a limit.
\end{prop}

\begin{proof}Let $J$ denote the Jacobi matrix associated with the sequence $(b_n),$ acting on $\ell^2(\mathbb{N}_0)$ by
\begin{gather*}
J\delta_n=b_n\delta_{n+1}+b_{n-1}\delta_{n-1},\qquad n\ge 0,
\end{gather*}
where $\delta_n$, $n\in\mathbb N_0$ is the standard orthonormal basis in $\ell^2(\mathbb N_0)$ with the convention $b_{-1}=0,$ $\delta_{-1}=0$.
The operator $J$ is bounded on $\ell^2(\mathbb{N}_0)$ and its operator norm is given by
\begin{gather*}
\|J\|^2=\lim_{n\to\infty}s_{2n}^{1/n}.
\end{gather*}

Let $S$ denote the unilateral shift defined by $S\delta_n=\delta_{n+1}$. Define
\begin{gather*}
J_N={1\over N}\sum_{j=0}^{N-1}(S^*)^jJS^j.
\end{gather*}
Then
 \begin{gather*}
\|J_N\|\le \|J\|.
\end{gather*}
The operator $J_N$ is corresponds to the Jacobi matrix with coefficients
\begin{gather*}
b_n^{[N]}={1\over N}\sum_{j=0}^{N-1}b_{n+j}.
\end{gather*}
For any fixed $n$ we have
\begin{gather*}
\lim_{N\to\infty} \big[b_{n+1}^{[N]}-b_n^{[N]}\big]=\lim_{N\to\infty}{b_{n+N}-b_{n}\over N}=0,
\end{gather*}
hence
\begin{gather}\label{limit}
\lim_{N\to\infty} \big[b_{n}^{[N]}-b_0^{[N]}\big]=0.
\end{gather}
We will consider two cases.

{\bf Case (a).} The sequence $b_0^{[N]}$, $N\ge 1$ does not converge to $0$. Let $b$ denote the greatest accumulation point of $b_0^{[N]}$, $N\ge 1$. Then there exists a sequence $N_1<N_2<\cdots$ from $\N$ so that
\begin{gather*}
\lim_{{k\to\infty}}b_{0}^{[N_k]}=b.
\end{gather*}
 By (\ref{limit}) for any fixed $n$ we have
\begin{gather*}
\lim_{{k\to\infty}}b_{n}^{[N_k]}=b.
\end{gather*}
The latter implies that the sequence of the operators $J_{N_k}$ converges $*$-weakly to the operator $\widetilde{J}$ with constant Jacobi coefficients $b$. It is well known that
\begin{gather*}
\big\|\widetilde{J}\big\|=2b.
\end{gather*}
 Thus, by Fatou's lemma, we get
\begin{gather}\label{norm}
2b=\big\|\widetilde{J}\big\|\le \liminf_{k\to\infty} \|J_{N_k}\|\le \|J\|.
\end{gather}
On the other hand we have
\begin{gather}\label{ineq}
(b_0b_1\cdots b_{N-1})^{1/ N}\le {1\over N}\sum_{j=0}^{N-1}b_j=b_0^{[N]}.
\end{gather}
Hence
\begin{gather}\label{ga}
\limsup_N (b_0b_1\cdots b_{N-1})^{1/ N}\le b.
\end{gather}
Therefore by (\ref{norm}) and (\ref{ga}) we obtain
\begin{gather*}\liminf_{N\to\infty} {{s_{2N}}^{1/N}\over \big(b_0^2b_1^2\cdots b_{N-1}^2\big)^{1/ N}}=
\liminf_{N\to\infty}{\|J\|^2\over \big(b_0^2b_1^2\cdots b_{N-1}^2\big)^{1/ N}}\\
\hphantom{\liminf_{N\to\infty} {{s_{2N}}^{1/N}\over \big(b_0^2b_1^2\cdots b_{N-1}^2\big)^{1/ N}}}{}
= {\|J\|^2\over \limsup\limits_{N\to\infty} \big(b_0^2b_1^2\cdots b_{N-1}^2\big)^{1/ N}}\ge {4b^2\over b^2}=4.
\end{gather*}

{\bf Case (b).} The sequence $b_0^{[N]}$, $N\ge 1$ converges to $0$. Then by (\ref{ineq}) we get
\begin{gather*}
\lim_{N\to\infty}(b_0b_1\cdots b_{N-1})^{1/ N}=0.
\end{gather*}
 Next
\begin{gather*}
\lim_{N\to\infty}{s_{2N}^{1/N}\over \big(b_0^2b_1^2\cdots b_{N-1}^2\big)^{1/ N}}=\lim_{N\to\infty}{\|J\|^2\over \big(b_0^2b^2_1\cdots b^2_{N-1}\big)^{1/ N}}=\infty.
\end{gather*}
The last part of the result follows by inspection of the proof given for the first part.
\end{proof}

\begin{rem}\label{thm:Geg}\rm The proof of Proposition~\ref{mean} shows that if $\lim b_n=1/2$ and if the moment representing measure $\mu$ is supported by $[-1,1]$, then $\|J\|=1$ and $\lim U_n^{1/n}=4$. Examples of this are given by the Gegenbauer weights
\begin{gather*}
w_\lambda(x)=\frac{(1-x^2)^{\lambda-1/2}}{B(1/2,\lambda+1/2)},\qquad -1<x<1,\qquad \lambda>-1/2.
\end{gather*}
It is known that
\begin{gather*}
s_{2n}=\frac{(1/2)_n}{(\lambda+1)_n},\qquad b_n=\frac12\sqrt{\frac{(n+1)(n+2\lambda)}{(n+\lambda)(n+\lambda+1)}},
\end{gather*}
where the expression for $b_n$ shall be interpreted for $\lambda=0$ as $b_0=1/\sqrt{2}$, $b_n=1/2$, $n \ge 1$, cf.~\cite{K:S}, so it is easy to see that $\lim U_n^{1/n}=4$ independent of $\lambda$.
\end{rem}

\subsection[Non-unicity of the inverse of $\mathcal H$]{Non-unicity of the inverse of $\boldsymbol{\mathcal H}$}

A positive definite Hankel matrix $\mathcal H$ acts injectively on $\mathcal F_c(\N_0)$ in the sense that $\mathcal Hv=0$ for $v\in \mathcal F_c(\N_0)$ implies $v=0$.

Nevertheless, it is possible that there exist sequences $v\neq 0$ with $\mathcal H v=0$, and such that $\mathcal H v$ is absolutely convergent in the sense that $\sum_n |s_{m+n}v_n| <\infty$ for all $m\ge 0$.

We show this for the following variant of the log-normal moment sequence. Let $0<q<1$ and consider
\begin{gather}\label{eq:logn1}
s_n=q^{-n^2/2}=\int_0^\infty x^n w_q(x)/x{\rm d}x,\qquad n\ge 0,
\end{gather}
where
\begin{gather*}
 w_q(x)=\frac{1}{\sqrt{2\pi\log(1/q)}} \exp\left(-\frac{(\log x)^2}{2\log(1/q)}\right).
\end{gather*}
Note that $w_q(x/q)=\sqrt{q} w_q(x)/x$, so from \cite[equation~(42)]{B:S1}, we get that the orthonormal polynomials for~\eqref{eq:logn1} are
\begin{gather}\label{eq:logn3}
P_n(x)=(-1)^n\frac{q^{n/2}}{\sqrt{(q;q)_n}}\sum_{k=0}^n\binom{n}{k}_q (-1)^kq^{k^2-k/2} x^k,
\end{gather}
with
\begin{gather*}
\binom{n}{k}_q=\frac{(q;q)_n}{(q;q)_k (q;q)_{n-k}}.
\end{gather*}

\begin{thm}\label{thm:nonuniq1} Let $\mathcal H$ be the Hankel matrix for the indeterminate Stieltjes problem with moments~\eqref{eq:logn1}. The set $\mathcal V$ of real sequences $v$ such that $\mathcal H v=0$ and $\mathcal H v$ is absolutely convergent, is an infinite dimensional vector space. For any $k\ge 1$ and $(a_0,a_1,\ldots, a_{k-1})\in\R^k$ there exists $v\in\mathcal V$ such that $v_j=a_j$, $j=0,1,\ldots, k-1$.
\end{thm}

\begin{proof} Let $v=(v_n)_{n\ge 0}$ be defined by
\begin{gather}\label{eq:logn4}
v_n=(-1)^n\frac{q^{n^2-n/2}}{(q;q)_n},\qquad n\ge 0.
\end{gather}
It is well-known that
\begin{gather*}
(z;q)_\infty=\sum_{n=0}^\infty (-1)^n\frac{q^{\binom{n}{2}}}{(q;q)_n} z^n,\qquad z\in\C,
\end{gather*}
see~\cite{G:R}. For $m\ge 0$ we therefore get
\begin{gather*}
\sum_{n=0}^\infty s_{m+n}v_n=q^{-m^2/2}\sum_{n=0}^\infty (-1)^n\frac{q^{\binom{n}{2}-mn}}{(q;q)_n}=q^{-m^2/2}\big(q^{-m};q\big)_\infty=0,
\end{gather*}
and
\begin{gather*}
\sum_{n=0}^\infty s_{m+n}|v_n|=q^{-m^2/2}\sum_{n=0}^\infty \frac{q^{\binom{n}{2}-mn}}{(q;q)_n}=q^{-m^2/2}\big({-}q^{-m};q\big)_\infty<\infty,
\end{gather*}
showing that $v\in \mathcal V$ and $v_0=1$.

We next construct $v\in\mathcal V$ with $k$ initial values and this implies that $\dim\mathcal V=\infty$.

For $(c_0,\ldots, c_{k-1})\in\R^k$ to be chosen later, we define the entire function
\begin{gather*}
g_k(z)=\left(\sum_{j=0}^{k-1} c_jz^j\right) (z;q)_\infty=\sum_{n=0}^\infty b_nz^n,
\end{gather*}
where
\begin{gather*}
b_n=\sum_{j=0}^{n\wedge (k-1)} c_j(-1)^{n-j}\frac{q^{\binom{n-j}{2}}}{(q;q)_{n-j}},\qquad n\ge 0.
\end{gather*}
Defining $v_n=q^{n^2/2}b_n$, $n\ge 0$, we get for $m\ge 0$
\begin{gather*}
\sum_{n=0}^\infty s_{m+n}v_n=q^{-m^2/2}\sum_{n=0}^\infty b_n q^{-mn}=q^{-m^2/2}g_k\big(q^{-m}\big)=0.
\end{gather*}
To see the absolute convergence, we write
\begin{gather*}
 \sum_{n=k-1}^\infty s_{m+n}|v_n| \le q^{-m^2/2}\sum_{n=k-1}^\infty q^{-mn}\sum_{j=0}^{k-1}|c_j|\frac{q^{\binom{n-j}{2}}}{(q;q)_{n-j}}
 \le q^{-m^2/2}\sum_{j=0}^{k-1}|c_j| \sum_{n=j}^\infty q^{-mn}\frac{q^{\binom{n-j}{2}}}{(q;q)_{n-j}}\\
\hphantom{\sum_{n=k-1}^\infty s_{m+n}|v_n| }{} = q^{-m^2/2}\sum_{j=0}^{k-1}|c_j| q^{-mj}\sum_{n=0}^\infty q^{-mn}\frac{q^{\binom{n}{2}}}{(q;q)_{n}}<\infty.
\end{gather*}
For $i=0,1,\ldots,k-1$ we shall solve the linear equations
\begin{gather*}
v_i=q^{i^2/2}\sum_{j=0}^{i} c_j (-1)^{i-j}\frac{q^{\binom{i-j}{2}}}{(q;q)_{i-j}}=a_i,
\end{gather*}
which have a unique solution $(c_0,\ldots,c_{k-1})$.
\end{proof}

By \eqref{eq:logn3} we see that the matrix $\mathcal A=\{a_{j,k}\}$ is given for $j\ge k$ as
\begin{gather*}
a_{j,k} = (-1)^{j+k}q^{j^2+k^2-(j+k)/2}\sum_{n=j}^\infty \frac{\binom{n}{j}_q\binom{n}{k}_q}{(q;q)_n}q^{n}\\
\hphantom{a_{j,k}}{} = (-1)^{j+k}\frac{q^{j^2+k^2-(j+k)/2}}{(q;q)_j(q;q)_k}\sum_{n=j}^\infty \frac{(q;q)_n}{(q;q)_{n-j}(q;q)_{n-k}} q^n\\
\hphantom{a_{j,k}}{} = (-1)^{j+k}\frac{q^{j^2+k^2+(j-k)/2}}{(q;q)_{j-k}(q;q)_k}\sum_{p=0}^\infty
\frac{(q^{j+1};q)_p}{(q;q)_p (q^{j-k+1};q)_p}q^p,
\end{gather*}
hence
\begin{gather*}
|a_{j,k}|\le \frac{q^{j^2+k^2}}{(q;q)_j(q;q)_k(q;q)_\infty}\sum_{p=0}^\infty \frac{q^p}{(q;q)_p}.
\end{gather*}
The last sum equals $(q;q)_\infty^{-1}$, and by symmetry we get{\samepage
\begin{gather*}
|a_{j,k}|\le \frac{q^{j^2+k^2}}{(q;q)_j(q;q)_k(q;q)_\infty^2},\qquad j,k\ge 0.
\end{gather*}
 This shows that $\mathcal A\mathcal H$ is absolutely convergent, and it is easy to verify that property (cs*) holds.}

For $v\in\mathcal V\setminus\{0\}$ we have
\begin{gather*}
v=(\mathcal A\mathcal H)v\neq \mathcal A(\mathcal H v)=0.
\end{gather*}

\begin{rem}\label{thm:nonuniq2}There are symmetric matrices $\widetilde{\mathcal A}$ different from $\mathcal A$ which satisfy $\widetilde{\mathcal A}\mathcal H=\mathcal I$.

We simply construct a non-zero symmetric matrix $\mathcal M$ such that $\mathcal H\mathcal M=0$ and such that the product is absolutely convergent, and we then define $\widetilde{\mathcal A}=\mathcal A+\mathcal M$.

As the first column of $\mathcal M$ we can take $v$ defined by \eqref{eq:logn4}. As the second column we can choose $v^{(1)}\in\mathcal V$ with the initial condition $v^{(1)}_0=v_1$. As the third column we choose $v^{(2)}\in\mathcal V$, now with two initial conditions in order to make $\mathcal M$ symmetric and so on.
\end{rem}

\subsection[Behaviour of the functions $u$, $v$ at infinity]{Behaviour of the functions $\boldsymbol{u}$, $\boldsymbol{v}$ at infinity}

\begin{prop} The function $u$ given in \eqref{eq:ua} satisfies
\begin{gather}\label{eq:uainf}
\lim_{\alpha\to\infty}\alpha u(\alpha)=\frac{\pi^2}{24},
\end{gather}
and the function $v$ given by \eqref{eq:va} satisfies
\begin{gather}\label{eq:vainf}
\lim_{\alpha\to\infty}\alpha v(\alpha)=\frac{\pi^2}{6}.
\end{gather}
\end{prop}

\begin{proof} From \eqref{eq:ua} we get
\begin{gather*}
u(\alpha)=\max_{2^{-1}\le x \le 1}\big(x^{1/\alpha}+1\big)\int_{x^{1/\alpha}}^1\log\left(\frac{y^\alpha}{1-y^\alpha}\right)\frac{{\rm d}y}{(1+y^2)}.
\end{gather*}
After the substitution $t=y^\alpha$ in the integral we find for fixed $\alpha>0$
\begin{gather*}
\alpha u(\alpha) = \max_{2^{-1}\le x \le 1}\big(x^{1/\alpha}+1\big)\int_x^1\log\left(\frac{t}{1-t}\right)
\frac{t^{1/\alpha}}{(1+t^{1/\alpha})^2}\frac{{\rm d}t}{t}\\
\hphantom{\alpha u(\alpha)}{} \le \frac12\int_{1/2}^1\log\left(\frac{t}{1-t}\right)\frac{{\rm d}t}{t}=\frac{\pi^2}{24}.
\end{gather*}
For the evaluation of the last integral we use that
\begin{gather*}
\frac{1}{2}\log^2(t) + \sum_{k=1}^\infty\frac{t^k}{k^2}
\end{gather*}
is a primitive of the integrand, and we next use formulas (0.233.3) and (0.241.2) from~\cite{Gr:Ry}.

On the other hand for any $x\in \big[2^{-1},1\big]$ we have
\begin{gather*}
\alpha u(\alpha) \ge \big(x^{1/\alpha}+1\big)\int_x^1\log\left(\frac{t}{1-t}\right)
\frac{t^{1/\alpha}}{(1+t^{1/\alpha})^2}\frac{{\rm d}t}{t}\\
 \hphantom{\alpha u(\alpha)}{} \ge \big(2^{-1/\alpha}+1\big)\int_{x}^1\log\left(\frac{t}{1-t}\right)
\frac{t^{1/\alpha}}{(1+t^{1/\alpha})^2}\frac{{\rm d}t}{t},
\end{gather*}
and hence
\begin{gather*}
\alpha u(\alpha)\ge \big(2^{-1/\alpha}+1\big)\int_{1/2}^1\log\left(\frac{t}{1-t}\right)
\frac{t^{1/\alpha}}{(1+t^{1/\alpha})^2}\frac{{\rm d}t}{t}.
\end{gather*}
This gives
\begin{gather*}
\liminf_{\alpha\to\infty} \alpha u(\alpha)\ge \frac12\int_{1/2}^1\log\left(\frac{t}{1-t}\right)\frac{{\rm d}t}{t},
\end{gather*}
so \eqref{eq:uainf} follows.

The substitution $t=x^\alpha$ in the expression \eqref{eq:va} for $v$ yields
\begin{gather*}
\alpha v(\alpha)=\int_1^\infty \log\big(\big[1-t^{-1}\big]^{-1}\big) t^{1/\alpha}\frac{{\rm d}t}{t},
\end{gather*}
and hence
\begin{gather*}
\lim_{\alpha\to\infty} \alpha v(\alpha)=\int_1^\infty\log\big(\big[1-t^{-1}\big]^{-1}\big)\frac{{\rm d}t}{t}=
-\int_0^1 \log(1-t)\frac{{\rm d}t}{t}=\frac{\pi^2}{6},
\end{gather*}
which shows \eqref{eq:vainf}.
\end{proof}

\subsection*{Acknowledgements}
\vspace{-1.5mm}

We are grateful to all the referees. Their remarks improved the exposition substantially.

\vspace{-2.5mm}

\pdfbookmark[1]{References}{ref}
\LastPageEnding

\end{document}